\documentclass[11pt,oneside,english]{amsart}
\usepackage[T1]{fontenc}
\usepackage[latin9]{inputenc}
\usepackage{geometry}
\usepackage{babel}
\usepackage{amstext}
\usepackage{amsthm}
\usepackage{amssymb}
\usepackage{xcolor}
\usepackage[unicode=true,
 bookmarks=true,bookmarksnumbered=false,bookmarksopen=false,
 breaklinks=false,pdfborder={0 0 1},backref=false,colorlinks=false]
 {hyperref}
 \usepackage{graphicx}
 \usepackage{caption}
 \usepackage{float}
\usepackage[caption = false]{subfig}
\usepackage{soul}

\makeatletter
\numberwithin{equation}{section}
\numberwithin{figure}{section}
\theoremstyle{plain}
\newtheorem{thm}{\protect\theoremname}[section]
\theoremstyle{plain}
\newtheorem{prop}[thm]{\protect\propositionname}
\theoremstyle{plain}
\newtheorem{lem}[thm]{\protect\lemmaname}

\IfFileExists{lmodern.sty}{\usepackage{lmodern}}{}

\makeatother

\providecommand{\lemmaname}{Lemma}
\providecommand{\propositionname}{Proposition}
\providecommand{\theoremname}{Theorem}
\providecommand{\corollaryname}{Corollary}

\newcommand{\ve}{\varepsilon}
\DeclareMathOperator{\rad}{\textup{rad}}

\begin{document}
\title[A.s.\ asymptotics for number variance of dilated integer sequences]{Almost sure asymptotics for the number variance of dilations of integer sequences}
\author{Christoph Aistleitner}
\address{Graz University of Technology, Institute of Analysis and Number Theory, Steyrergasse 30, 8010 Graz, Austria}
\email{aistleitner@math.tugraz.at}
\author{Nadav Yesha}
\address{Department of Mathematics, University of Haifa, 3103301 Haifa, Israel}
\email{nyesha@univ.haifa.ac.il}

\begin{abstract}
Let $(x_n)_{n=1}^\infty$ be a sequence of integers. We study the number variance of dilations $(\alpha x_n)_{n=1}^\infty$ modulo 1 in intervals of length $S$, and establish pseudorandom (Poissonian) behavior for Lebesgue-almost all $\alpha$ throughout a large range of $S$, subject to certain regularity assumptions imposed upon  $(x_n)_{n=1}^\infty$. For the important special case $x_n = p(n)$, where $p$ is a polynomial with integer coefficients of degree at least 2, we prove that the number variance is Poissonian for almost all $\alpha$ throughout the range $0 \leq S \leq (\log N)^{-c}$, for a suitable absolute constant $c>0$. For more general sequences $(x_n)_{n=1}^\infty$, we give a criterion for Poissonian behavior for generic $\alpha$ which is formulated in terms of the additive energy of the finite truncations $(x_n)_{n=1}^N$.
\end{abstract}

\maketitle

\section{Introduction}
Let $\left(x_{n}\right)_{n=1}^{\infty}$ be a sequence of integers. Our goal in this paper is to study the fluctuation of the number of
elements modulo 1 of dilations $\left(\alpha x_{n}\right)_{n=1}^{\infty}$
in intervals of length $S$, for generic
values of $\alpha$. More precisely, let
\begin{equation} \label{sn_s_def}
S_{N}\left(S\right) = S_{N}\left(S,y\right) =\sum_{n=1}^{N}\sum_{j\in\mathbb{Z}} \mathbf{1}_{[-S/2,S/2)}\left(\alpha x_{n}-y+j\right)
\end{equation}
denote the number of elements of the sequence (mod 1) in an interval of length $0\le S \le 1$ around $y\in\left[0,1\right]$.
Then the expected value of $S_{N}\left(S\right)$ with respect to $y$ is clearly
\[
\int_{0}^{1}S_{N}\left(S,y\right)\,dy=NS,
\]
and the classical theory of uniform distribution modulo 1, going back to Weyl's seminal paper of 1916 \cite{weyl}, tells us that for any sequence $\left(x_{n}\right)_{n=1}^{\infty}$ of distinct integers for Lebesgue-almost all $\alpha$ we have
$$
S_N(S,y) - NS = o(N), \qquad \text{uniformly for $0 \leq S \leq 1$, as $N \to \infty$}.
$$
The ``\emph{number variance}'' $V = V(N,S,\alpha)$ of the first $N$ elements of $(\alpha x_n)_{n=1}^\infty$ mod 1 is defined as the variance of $S_N(S,y)$ with respect to $y$, i.e.\
\begin{equation*}
V(N,S,\alpha) :=\int_{0}^{1}S_{N}\left(S,y\right)^{2}\,dy-N^{2}S^{2}.
\end{equation*}
 The number variance is, among other statistics such as pair correlation and nearest neighbor spacings, a popular test for pseudorandom behavior of real-valued sequences, in particular in the context of theoretical physics (see for example \cite{bogia}). In our setup, pseudorandom (``Poissonian'') behavior means that
$$
V=NS+o\left(NS\right) \qquad \text{as $N\to\infty$},
$$
uniformly throughout a large range of $S$, and it is our aim to show that indeed this is the case for generic $\alpha$, subject to some regularity conditions on $(x_n)_{n=1}^\infty$.\\

Of particular interest is the case when $x_n = p(n)$ for an integer-valued polynomial $p$. Sequences of the form $(\alpha p(n))_{n=1}^\infty$ have a long and rich history in analytic number theory and ergodic theory. The quadratic case is of special interest in the context of theoretical physics, since the sequence $(\alpha n^2)_{n=1}^\infty$ can be interpreted as modeling the energy levels of the ``boxed oscillator'' in the high-energy limit, and thus serving as a simple test case of the far-reaching Berry--Tabor conjecture in quantum chaology; see \mbox{\cite{bt,rs}}. We prove that the number variance of $(\alpha p(n))_{n=1}^\infty$ is Poissonian for a wide range of the parameter $S$, for any polynomial in $\mathbb{Z}[x]$ of degree at least 2.

\begin{thm}
\label{thm:Main_thm}Let $p(x) \in \mathbb{Z}[x]$ be a polynomial of degree at least $2$, and let $x_n = p(n),\, n \geq 1$. Then there exists a constant $c>0$ such that for the sequence $\left(\alpha x_n \right)_{n=1}^\infty$ for almost all $\alpha\in\left[0,1\right]$ we have
\begin{equation}
V=NS + o(NS),
\end{equation}
uniformly throughout the range $0\le S\le (\log N)^{-c}$, as $N\to\infty$.
\end{thm}

Theorem \ref{thm:Main_thm} (and more abstractly, Theorem \ref{thm:2} below) significantly improves the analogous result in \cite{ly} which only holds (in a non-uniform fashion)
for $ S= N^{-\beta}$, $ 1/2 < \beta \le 1 $.
As the proofs show, the problem becomes more difficult as $S$ increases, as a consequence of the fact that the variance of $V$ with respect to $\alpha$ increases with $S$. A much shorter proof could be given if the admissible range for $S$ was reduced to $0 \leq S \leq N^{-\ve}$ for $\varepsilon>0$, rather than the range given in the statement of the theorem. The significance of the conclusion of Theorem \ref{thm:Main_thm} is that we believe $(\log N)^{-c}$ to be the optimal upper endpoint for the admissible range for $S$ when studying dilations of polynomial sequences (for generic values of $\alpha$), except for the specific optimal value of $c>0$ which remains unknown. This is particularly plausible in the case of quadratic polynomials, where the sum of squares of the representation numbers, which will be seen to control an important aspect of this problem, grows asymptotically as $\sum_{u \leq x} r (u)^2 \approx x (\log x)^3$ (see Lemma \ref{lemma_blogra} below), while for ``random'' behavior the growth rate should rather be linear in $x$. Thus the variance (with respect to $\alpha$) of the number variance is too large by a factor of logarithmic order, and this should lead to a logarithmic loss in the maximal admissible range for $S$. In the case of polynomials of degree 3 and higher, the situation is less clear (the sum of squares of the representation numbers has the ``correct'' asymptotic order, see Lemma \ref{lemma_rep_deg_3}), and it is possible that $S \leq (\log N)^{-\varepsilon}$ for arbitrary $\varepsilon > 0$ can be reached in this case; roughly speaking, the answer to the question whether such an improvement is possible or not will depend on the decay rate of the ``tail probabilities'' which quantify the measure of those $\alpha$ for which $V(N,S,\alpha)$ deviates significantly from its expected value, and it is clear that the $L^2$ methods from the present paper (which always entail the loss of a factor of logarithmic size from an application of Chebyshev's inequality and the Borel--Cantelli lemma) as a matter of principle cannot reach as far as $S \approx (\log N)^{-\varepsilon}$. The proof of Theorem \ref{thm:Main_thm} will show that the values $c=32+\ve$ for any $\ve>0$ (for $p$ of degree 2) and $c = 3 + \ve$ for any $\ve>0$ (for $p$ of degree 3 and higher) are admissible in Theorem \ref{thm:Main_thm}. With some additional effort these numerical values for $c$ could most likely be improved, but we have not made an effort to optimize them.\\

Note that the assumption of $p$ having degree at least 2 is important for the correctness of Theorem \ref{thm:Main_thm}. In the case of a polynomial of degree 1 one studies, without loss of generality, the sequence $(\alpha n)_{n=1}^\infty$ mod 1, and the distribution of this sequence is ``too regular'' in comparison with a random sequence (and instead has a very rigid structure, which in quantum chaology is sometimes called a ``picket fence'' pattern). More precisely, while the number variance $V(N,S,\alpha)$ clearly blows up for rational $\alpha$, for all irrational $\alpha$ it can become as small as $\mathcal{O}(1)$ (rather than asymptotic order $NS$) for infinitely many $N$. We state this observation as a proposition.

\begin{prop} \label{prop1}
Consider the sequence $(\alpha n)_{n=1}^\infty$. Then for all irrational $\alpha$ there are infinitely many $N$ for which $V(N,S,\alpha) \leq 9$, uniformly in $0\le S \le 1$. Consequently, the number variance of $\left(\alpha x_n \right)_{n=1}^\infty$ is not Poissonian as $NS \to \infty$.
\end{prop}

From a technical perspective, the key properties of polynomial sequences which allow us to establish Theorem \ref{thm:Main_thm} are a) having control over the number of solutions $n_1,n_2,n_3,n_4$ of the equation $p(n_1) - p(n_2) = p(n_3) - p(n_4)$, and b) having control over the divisor structure of the set of differences $p(m) - p(n)$. We can also prove an abstract theorem for general sequences $(x_n)_{n =1}^\infty$, which utilizes control over the arithmetic structure of the differences $x_n - x_m$, $1 \leq m,n \leq N$, in order to obtain an almost sure asymptotic estimate for the number variance up to a suitable size of the length parameter $S$. Here the admissible range for $S$ has to be balanced with the extent of arithmetic control which we have over the sequence $(x_n)_{n =1}^\infty$. In the statement of the theorem, for a given sequence $\left(x_{n}\right)_{n=1}^{\infty}$ we write
$$
E_N = \# \left\{(n_1,n_2,n_3,n_4) \in \{1,\dots,N\}^4:~x_{n_1} - x_{n_2} = x_{n_3} - x_{n_4} \right\}
$$
for the additive energy of the truncated sequence $(x_n)_{n=1}^N$.

\begin{thm}
\label{thm:2}
Let $\ve>0$ be arbitrary. Let $\left(x_{n}\right)_{n=1}^{\infty}$ be a sequence of distinct positive integers. Assume that there exists an $\eta>0$ such that
\begin{equation} \label{eq_dioph_cond_variant}
E_N = \mathcal{O}_\eta(  N^{2+\eta})
\end{equation}
as $N \to \infty $. Then for almost all $\alpha \in [0,1]$ we have
\begin{equation}
V=NS + o(NS),
\end{equation}
uniformly in the range  $0\le S\le N^{-\eta-\ve}$, as $N\to\infty$.
\end{thm}

We believe that the relation between condition \eqref{eq_dioph_cond_variant} and the range for $S$ in the conclusion of Theorem \ref{thm:2} is essentially optimal, in the sense that under the same assumption the range for $S$ cannot be extended (in general) to the slightly wider range $0 \leq S \leq N^{-\eta+\varepsilon}$. However, we have not been able to prove this.\\

Theorem \mbox{\ref{thm:2}} stands in a line with earlier results which express the effect of phenomena from additive combinatorics on the metric theory of dilated integer sequences. Such a relation has been observed most prominently in the context of pair correlation problems (see e.g. \mbox{\cite{all,bdm,bw,ls1}}), but also in discrepancy theory \mbox{\cite{al}} and in the metric theory of minimal gaps \mbox{\cite{aem_d,regav, rud}}.\\

%\begin{thm}
%\label{thm:3} Let $\eta \in (0,1)$ be arbitrary. There exists a sequence $\left(x_{n}\right)_{n=1}^{\infty}$ of positive integers, such that
%\begin{equation} \label{eq_dioph_cond}
%\sum_{u \geq 1} \left(\textup{Rep}_{1,N} (u)\right)^2  \leq N^{2+\eta},
%\end{equation}
%and such that for almost all $\alpha \in [0,1]$ there are infinitely many values of $N$ for which [the desired convergence fails]\footnote{\hl{Nadav: I haven't thought about it yet.}}
%\end{thm}

It would be interesting to compare the results in this paper to the corresponding results for the number variance of a sequence of random points in the unit interval. However, we have not been able to find a satisfactory result for the random case in the literature, nor have we been able to establish a complete description of the (almost sure) behavior of the number variance in the random case ourselves. Of course, for the number variance of random points one expects $V(N,S) = NS + o(NS)$ as $N \to \infty$ almost surely, throughout a wide range of $S$, and using the methods from the present paper (in a much simpler form, since the variance estimate becomes much easier in the random case) one could show that this is indeed the case uniformly in the range $0 \leq S \leq (\log N)^{-\kappa}$ for some suitable $\kappa>0$.  One would expect that the admissible range for $S$ is actually significantly larger in the random case. However, it is remarkable that it turns out that even in the random case the asymptotics $V(N,S) = NS (1+o(1))$ does not hold true all the way up to $S = o(1)$, almost surely; instead, there must exist a threshold value (expressible as a function of $N$) such that this asymptotics can fail to be true for values of $S = S(N)$ which are above the threshold. This critical threshold value for $S$ seems to be around $(\log \log N)^{-1}$, but we have not been able to establish this precisely.  What we can prove is the following.

\begin{prop} \label{prop2}
Let $(X_n)_{n=1}^\infty$ be a sequence of independent, identically distributed random variables having uniform distribution on $[0,1]$. Then almost surely there exist infinitely many $N$ such that for the specific value $S = S(N) = 35 (\log \log N)^{-1}$ the number variance satisfies
$$
V(N,S) \geq 1.001 NS.
$$

\end{prop}

Informally speaking this means that, almost surely, the range of $S$ throughout which the asymptotics $V(N,S) = NS + o(NS)$ holds true cannot be larger than $0 \leq S \leq 35 (\log \log N)^{-1}$.

In the language of probability theory, counting functions such as those in \eqref{sn_s_def} compare the empirical distribution of the random sample $X_1, \dots, X_N$ to the underlying probability distribution (in our case the uniform distribution), giving essentially what is called the ``empirical process'' (see \cite{vdv} for the basic theory of empirical processes).  A keystone result in this area is the Koml\'{o}s--Major--Tusn\'{a}dy (KMT) theorem \cite{kom}, which states that for i.i.d. uniformly distributed random variables on $[0,1]$  there is a Brownian bridge $B(t)$ such that
$$
\sum_{n=1}^N \mathbf{1}_{[0,t)} (X_n) - t N  = \sqrt{N} B(t), \qquad \text{uniformly in $t$},
$$
with high probability  (a Brownian bridge is essentially a Brownian motion on $[0,1]$, which is conditioned to satisfy the additional requirement that $B(1)=0$). As a consequence one can establish that the distribution of the number variance $V(N,S)$ of $X_1, \dots, X_N$ is, with high probability and up to very small errors, the same as that of
\begin{equation} \label{brown}
N \int_0^1 \left( B(t + S) - B(t) \right)^2 dt,
\end{equation}
where the argument $t+S$ has to be read modulo 1. The problem is that the exact  distribution of \eqref{brown} seems to be unknown, and we have not been able to calculate it precisely ourselves. Understanding this distribution in details should allow to find the precise threshold value of $S$ where $V(N,S) \sim NS$ fails to be true almost surely. More details on the KMT approximation are contained in Section \ref{sec_final}, where we use it to prove Proposition \ref{prop2}.\\

Once the behavior of the number variance of random points is understood, it would be interesting to compare the threshold value for $S$ in the random case to a corresponding result for $(\alpha x_n)_{n=1}^\infty$ for the case of \emph{lacunary} (that is, quickly increasing) integer sequences $(x_n)_{n=1}^\infty$. It is a well-known fact that the typical behavior of dilated lacunary sequences is often similar to that of true random sequences -- this observation can already be anticipated in Borel's \cite{borel} work on normal numbers, and is connected with names such as Salem, Zygmund, Kac, Erd\H os, and many others. For recent results on dilations of lacunary sequences in the framework of pair correlation problems and other ``local statistics'' problems, see for example \cite{abty,bpt,ct,rzlac,stef,yesha1}. The methods of the present paper would allow to establish that the number variance of dilated lacunary sequences is Poissonian for $S$ in the range $0 \leq S \leq (\log N)^{-\kappa}$ for some suitable $\kappa>0$, but a significant extension of this range (towards $S \approx (\log \log N)^{-1}$) would require a more sophisticated approach.\\

The plan for the rest of this paper is as follows. In Section \ref{sec_prep} we prepare the $L^2$ (variance) estimate, and establish the connection with the additive energy of truncations of the sequence $(x_n)_{n=1}^\infty$. In Section \ref{sec_proof_thm} we give the proof of Theorem \ref{thm:2}, which utilizes general bounds for GCD sums. In Section \ref{sec_poly_deg2}  we establish bounds for the additive energy for truncations of sequences of the form $(p(n))_{n=1}^\infty$, in the case when $p$ is a quadratic polynomial, and give the proof of Theorem \ref{thm:Main_thm} in this case. The same is done in Section \ref{sec_poly} for the case of polynomials of degree at least 3. Finally, Section \ref{sec_final} contains the proof of Propositions \ref{prop1} and \ref{prop2}.

\section{Preparations} \label{sec_prep}

Our method proceeds by estimating the first and second moments of $V(N,S)$ with respect to $\alpha$, and then using Chebyshev's inequality and the convergence part of the Borel--Cantelli lemma. It is crucial that we want to establish a result which holds uniformly throughout a wide range of admissible values of $S$, and not just for one particular value of $S$ (as in the case of pair correlation problems, where $S$ is assumed to be inversely proportional with $N$ and the product $NS$ is assumed to remain constant as $N \to \infty$).\\

 Let $S_{N}\left(S\right)$ be as in \eqref{sn_s_def}. Taking the second moment with respect to $y$, one can calculate (the details are given in Section 2 of \cite{mark1}) that
\begin{gather}
\int_{0}^{1}S_{N}\left(S\right)^{2}\,dy=\sum_{m,n=1}^{N}\sum_{j\in\mathbb{Z}}\psi_{S/2}\left(\alpha x_m - \alpha x_n +j\right)\label{eq:second_moment}
\end{gather}
where
\[
\psi_{r} (t ) =\left( \mathbf{1}_{[-r,r)}*\mathbf{1}_{[-r,r)} \right) (t) =\int_{\mathbb{R}} \mathbf{1}_{[r,r)}\left(t-y\right) \mathbf{1}_{[-r,r)}\left(y\right)\,dy
\]
is obtained by convolution. For $r > 0$ the function $\psi_r$ can be written in terms of the ``tent map'' $\Delta(t) = \max\{1-|t|,0\}$ as
$$
\psi_r(t) = 2r \Delta \left( \frac{t}{2r} \right).
$$
We note that $\psi_ {S/2}$ is non-negative and that
$$
\|\psi_{S/2} \|_1 = S^2, \qquad \|\psi_{S/2} \|_2^2 = \frac{2 S^3}{3}, \qquad \|\psi_{S/2}\|_\infty = \psi_{S/2} (0) = S,
$$
where the norms are the $L^1$, $L^2$, $L^\infty$ norms on $[0,1]$. The terms $m=n$ in (\ref{eq:second_moment}) clearly contribute $NS$
to the sum. Thus, the variance of $S_N(S)$ (the ``\emph{number variance}'')
is equal to
\begin{equation}
V(N,S,\alpha) :=\int_{0}^{1}S_{N}\left(S\right)^{2}\,dy-N^{2}S^{2}=NS-N^{2}S^{2}+I_{\text{off}},\label{eq:variance_to_correlation}
\end{equation}
where
\begin{equation}  \label{eq:variance_to_correlation_2}
I_{\text{off}}=\sum_{m\ne n=1}^{N}\sum_{j\in\mathbb{Z}}\psi_{S/2}\left(\alpha x_m-\alpha x_n+j\right)
\end{equation}
is the contribution of the off-diagonal terms. The aim then is to show that
\begin{equation} \label{eq_num_var}
I_{\text{off}} = N^2 S^2 + o(NS)
\end{equation}
uniformly throughout a wide range of $S$, almost surely. The uniformity in $S$ is achieved by a dyadic decomposition of the admissible range for $S$, and using a union bound over the exceptional probabilities for all elements of this decomposition before the application of the Borel--Cantelli lemma.\\

One particular difficulty is that the asymptotics which we wish to establish, namely $I_{\text{off}} = N^2 S^2 + o(NS)$, has a very small error term with respect to $N$. For comparison, in pair correlation problems one tries to establish
$$
\sum_{m\ne n=1}^{N}\sum_{j\in\mathbb{Z}} \mathbf{1}_{[-S,S]} \left(\alpha x_m-\alpha x_n+j\right) = 2 S N^2 + o(S N^2),
$$
where $S$ is in a fixed relation with $N$ (for classical pair correlation problems $S = s/N$ with a constant $s$, while for ``wide range correlations'' $S = N^{-\beta}$ for some fixed $\beta$). In this case it suffices to establish the desired convergence along an exponentially growing subsequence $N \approx \theta^m,~m\geq1$, for small $\theta>1$, since then the result for all other $N$ follows from a simple sandwiching argument (this is carried out in detail for example in \cite{all}). In contrast, in the problem for the number variance \eqref{eq_num_var} we allow ourselves only a much smaller error term with respect to $N$ (namely $o(N)$ instead of $o(N^2)$, while the main term remains quadratic in $N$), which means that the fluctuations of the sum with respect to $N$ have to be controlled much more carefully. This could be done via dyadic decomposition of the index range (additionally to the dyadic decomposition of the range of $S$, which is necessary to achieve the desired uniformity in $S$). Instead of doing this ``by hand'', we use a maximal $L^2$ inequality in the spirit of the Rademacher--Menshov inequality.

\begin{lem}[{Special case of \cite[Corollary 3.1]{mss}}] \label{lemma_mss}
Let $X_1, \dots, X_N$ be random variables. Assume that there is a non-negative function $G(i,j)$ for which
$$
G(N_1,N_2) + G(N_2+1,N_3) \leq G(N_1,N_3) \quad \text{for all} \quad 1 \leq N_1 \leq N_2 < N_3 \leq N.
$$
Assume that
$$
\mathbb{E} \left( \left| \sum_{n=N_1}^{N_2} X_n \right|^2 \right) \leq G(N_1,N_2) \quad \text{for all} \quad 1 \leq N_1 \leq N_2 \leq N.
$$
Then
$$
\mathbb{E} \left( \max_{1 \leq M \leq N} \left|   \sum_{n=1}^M X_n \right|^2 \right) \leq G(1,N) (1 + \log_2 N)^2,
$$
where $\log_2$ is the logarithm in base 2.
\end{lem}

Throughout the rest of the section, $(x_n)_{n=1}^\infty$ is always a sequence of distinct positive integers. The symbols $\mathbb{E}$ and $\mathrm{Var}$ are always understood to be used for integration with respect to $\alpha$ and with respect to the Lebesgue measure on $[0,1]$.\\

Let $N$ be given.  For parameters $v,c$ in the ranges
\begin{equation} \label{eq:param_ranges}
v\ge 0, \quad  0  \le c < 2^{v},
\end{equation}
we define
\begin{equation}
Y_{n,v,c} =Y_{n,v,c}\left(\alpha\right)  := 2 \sum_{j\in\mathbb{Z}} \sum_{m < n} f_{v,c}\left(\alpha x_{n}-\alpha x_{m}+j\right)  - 2 (n-1) (2c+1) 2^{-2v} \label{eq:Y_vars}
\end{equation}
where
\[
f_{v,c}(x) = \left\{ \begin{array}{ll} 2^{-v}, & \text{for }0 \leq |x| \leq c 2^{-v},\\ 2^{-v} - (|x| - c 2^{-v}), & \text{for }c 2^{-v} \leq |x| <  (c+1) 2^{-v},\\ 0, & \text{for } |x| >  (c+1) 2^{-v}. \end{array} \right.
\]
These functions $f_{v,c}$ are continuous functions whose graph has the shape of a truncated pyramid; they are constructed in such a way that each tent function $\psi_{S/2}$ as above can be decomposed dyadically into a sum of such functions, according to the binary expansion of $S$ (as visualized in Figure \ref{fig:1}). Note that the tilted parts of the graph of the tent function $\psi_{S/2}$ have slope $\pm 1$, independent of the value of $S$, and that the slope of the tilted part of the graph of $f_{v,c}$ also has slope $\pm 1$, independent of $v$ and $c$. One can easily calculate that
\begin{equation} \label{f_mean}
\int f_{v,c}(x) ~dx = (2c+1) 2^{-2v},
\end{equation}
so that
\begin{equation} \label{eq:average0}
\mathbb{E} [Y_{n,v,c}]  = 0 \qquad \text{for all $n,v,c$}.
\end{equation}
%and
%\begin{equation} \label{eq:average1}
%\mathbb{E} [Y_{n,v,c}^2]  \ll c n^2 2^{-3v} \qquad \text{for all $n,v,c$}
%\end{equation}

\begin{figure}[ht!]
\subfloat{\includegraphics[width = 0.5 \linewidth]{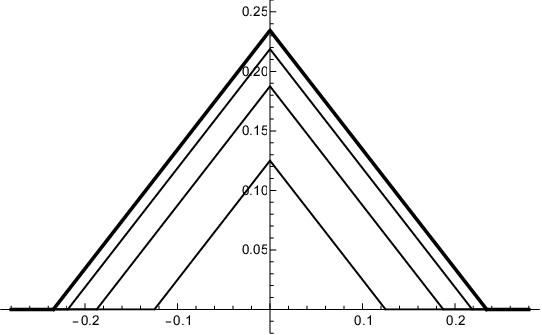}}
\subfloat{\includegraphics[width = 0.5 \linewidth]{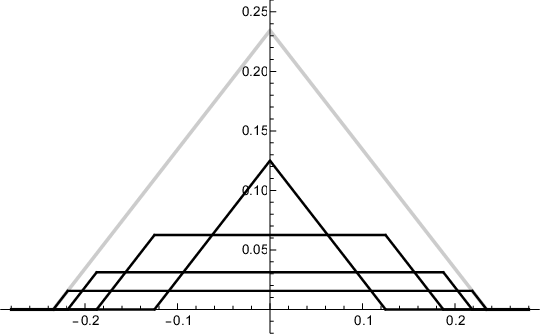}}
\caption{Dyadic decomposition of a tent function (left) and the way it decomposes into a sum of functions of the type $f_{v,c}$ (right). Pictured is the function $\psi_{S/2}$ for $S = \frac{15}{64} = \frac{1}{8} + \frac{1}{16} + \frac{1}{32} + \frac{1}{64}$. We have $\psi_{\frac{15}{64}} = f_{3,0} + f_{4,2} + f_{5,6} + f_{6,14}$. For this decomposition, a function $f_{v,c}$ for a certain value of $v$ is required exactly if the $v$-th digit (after the decimal point) in the binary expansion of $S$ is a ``1''. For such a $v$, the suitable value of $c$  can also be read off of the binary expansion of $S$, cf.\ Equation \eqref{eq_dyad_arise}. Note that all functions $f_{v,c}$ in this decomposition are continuous (actually, even Lipschitz continuous with parameter 1), which affects the decay order of the  Fourier coefficients and is crucially used for the variance estimate in Lemma \ref{prop:VarianceBound}--\ref{prop:VarianceBound_coarse_2}.}
\label{fig:1}
\end{figure}

\begin{lem} \label{prop:VarianceBound}
For $u\in \mathbb{N}$ let
\begin{equation} \label{def_rep}
\textup{Rep}_{N_1,N_2} (u) := \# \left\{1 \leq m < n,~N_1 \leq n \leq N_2:~|x_n - x_m| = u \right\}.
\end{equation}
Then for all $N_1$ and $N_2$ satisfying $1 \leq N_1 < N_2$, and for all $v,c$ as in \eqref{eq:param_ranges},  we have
\begin{eqnarray*}
& & \mathbb{E}\left[\left(\sum_{N_1 \leq n \leq N_2} Y_{n,v,c} \right)^2 \right] \\
& \ll &  2^{-3v} (v+1)  (c+1) \sum_{u_{1},u_{2} > 0}\textup{Rep}_{N_1,N_2}\left(u_{1}\right)\textup{Rep}_{N_1,N_2}\left(u_{2}\right)  \frac{(u_{1},u_{2})}{\max\{u_{1},u_{2}\}, }
\end{eqnarray*}
where the implied constant is independent of $N_1,N_2,v,c$.
\end{lem}

In the statement of the lemma, and in the sequel, $(\cdot,\cdot)$ denotes the greatest common divisor.

\begin{proof}[Proof of Lemma \ref{prop:VarianceBound}]
Our aim is to estimate
\begin{eqnarray}
\mathbb{E}\left[\left(\sum_{N_1 \leq n \leq N_2} Y_{n,v,c} \right)^2 \right] & =  & \int_{0}^{1}\left(\sum_{N_1 \leq n \leq N_2} Y_{n,v,c} (\alpha) \right)^{2}\,d\alpha.  \label{eq_var_1}
\end{eqnarray}
By the Poisson summation formula we have
\begin{eqnarray*}
Y_{n,v,c}\left(\alpha\right) = 2 \sum_{j\in\mathbb{Z}} \widehat{f_{v,c}}(j) \sum_{m < n} e^{2\pi i j (\alpha x_n -  \alpha x_m)} - 2 (n-1) (2c+1) 2^{-2v}.
\end{eqnarray*}

The function $f_{v,c}$ can be written as the convolution of the indicator function of the interval $\left[-(c+\frac{1}{2}) 2^{-v},(c+\frac{1}{2}) 2^{-v}\right]$ with the indicator function of $[-2^{-v-1}, 2^{-v-1}]$. Accordingly, the Fourier coefficients of $f_{v,c}$ can be calculated to be
\begin{align} \label{size_fourier_coeff}
\widehat{f_{v,c}}\left(j\right) & = \frac{\sin(2^{-v+1} (c+\frac{1}{2}) j \pi)}{j \pi} \frac{\sin(2^{-v} j \pi)}{j \pi}  , \qquad j \neq 0.
\end{align}
Note that all these Fourier coefficients are real and that $\widehat{f_{v,c}}\left(-j\right) = \widehat{f_{v,c}}\left(j\right)$. By \eqref{eq:average0} we have $\widehat{f_{v,c}}\left(0\right)=(2c+1)2^{-2v}$, so that

\begin{eqnarray*}
		Y_{n,v,c}\left(\alpha\right) = 2 \sum_{j\ne 0} \widehat{f_{v,c}}(j) \sum_{m < n} e^{2\pi i j (\alpha x_n -  \alpha x_m)}
\end{eqnarray*}
and therefore
\begin{eqnarray}
 \sum_{N_1 \leq n \leq N_2} Y_{n,v,c} (\alpha)  = 2 \sum_{u>0} \textup{Rep}_{N_1,N_2}\left(u\right)  \sum_{j\ne 0} \widehat{f_{v,c}}\left(j\right)e^{2\pi ij\alpha u}, \label{eq:Sum_Yn}
\end{eqnarray}
where the representation numbers $\textup{Rep}_{N_1,N_2}(u)$ are defined as in the statement of the lemma.

Thus, by the orthogonality of the trigonometric system we have
\begin{eqnarray}
& &  \int_0^1 \left( \sum_{N_1 \leq n \leq N_2} Y_{n,v,c} (\alpha)  \right)^2 d\alpha \nonumber\\
  & = & 4 \sum_{u_1,u_2 > 0}   \textup{Rep}_{N_1,N_2}\left(u_{1}\right)\textup{Rep}_{N_1,N_2}\left(u_{2}\right) \sum_{j_{1},j_{2}\ne0}\widehat{f_{v,c}}\left(j_{1}\right) \widehat{f_{v,c}}\left(j_{2}\right) \mathbf{1} \left( j_{1}  u_1 =j_{2} u_2 \right). \label{eq_ingred_2}
\end{eqnarray}
For given $u_1$ and $u_2$, the solutions $j_1,j_2$ of $j_1 u_1 = j_2 u_2$ are $j_1 = \frac{r u_2}{(u_1,u_2)}$ and $j_2 = \frac{r u_1}{(u_1,u_2)}$, for $r \in \mathbb{Z}$. Thus
\begin{equation} \label{eq_ingred_3}
\sum_{j_{1},j_{2}\ne0}\widehat{f_{v,c}}\left(j_{1}\right) \widehat{f_{v,c}}\left(j_{2}\right) \mathbf{1} \left( j_{1}  u_1 =j_{2} u_2 \right) \leq 2 \sum_{r \geq 1} \left| \widehat{f_{v,c}}\left(\frac{ru_{2}}{\left(u_{1},u_{2}\right)}\right) \widehat{f_{v,c}}\left(\frac{ru_{1}}{\left(u_{1},u_{2}\right)}\right) \right|.
\end{equation}
Now we claim that
\begin{equation} \label{eq_ingred_4}
\sum_{r \geq 1} \left| \widehat{f_{v,c}}\left(\frac{ru_{2}}{\left(u_{1},u_{2}\right)}\right) \widehat{f_{v,c}}\left(\frac{ru_{1}}{\left(u_{1},u_{2}\right)}\right) \right| \ll 2^{-3v} (v+1)(c+1) \frac{(u_{1},u_{2})}{\max\{u_{1},u_{2}\} }.
\end{equation}
Together with \eqref{eq_var_1}, \eqref{eq_ingred_2}, \eqref{eq_ingred_3}, this will yield the desired result. To see that \eqref{eq_ingred_4} indeed holds, we first use  \eqref{size_fourier_coeff}  to obtain
\begin{eqnarray}
& & \sum_{r \geq 1} \left| \widehat{f_{v,c}}\left(\frac{ru_{2}}{\left(u_{1},u_{2}\right)}\right) \widehat{f_{v,c}}\left(\frac{ru_{1}}{\left(u_{1},u_{2}\right)}\right) \right| \nonumber\\
& \ll & \sum_{r  \geq 1} \left|\sin \left(2^{-v} \frac{\pi r u_1}{(u_1,u_2)} \right)  \sin \left(2^{-v} \frac{\pi r u_2}{(u_1,u_2)} \right) \times \right. \label{eq:sum_here1} \\
& & \quad \left. \times \sin \left(2^{-v+1} \left(c+\frac{1}{2}\right) \frac{\pi r u_1}{(u_1,u_2)} \right) \sin \left(2^{-v+1} \left (c+\frac{1}{2}\right) \frac{\pi r u_2}{(u_1,u_2)} \right) \right| \frac{(u_1,u_2)^4}{r^4 u_1^2 u_2^2}. \label{eq:sum_here2}
\end{eqnarray}
Assume w.l.o.g.\ that $u_1 \leq u_2$. We will distinguish several cases, depending on whether it is better to use $|\sin(x)| \leq |x|$ or $|\sin(x)|\leq 1$ to estimate the size of the sum in lines \eqref{eq:sum_here1} and \eqref{eq:sum_here2}.\\

We first assume that we are in the case $2^{-v+1}\left (c+\frac{1}{2}\right) \frac{\pi u_1}{(u_1,u_2)} \geq 1$. Then the sum in lines \eqref{eq:sum_here1} and \eqref{eq:sum_here2} is bounded by
\begin{eqnarray*}
 & \ll & \sum_{r \geq 1} \left|\sin \left(2^{-v} \pi \frac{r u_1}{(u_1,u_2)} \right)  \sin \left(2^{-v} \pi \frac{r u_2}{(u_1,u_2)} \right) \right| \frac{(u_1,u_2)^4}{r^4 u_1^2 u_2^2} \\
 & \ll &  \sum_{r \geq 1} 2^{-2v} \frac{(u_1,u_2)^2}{r^2 u_1 u_2} \\
 & \ll & 2^{-2v} \frac{(u_1,u_2)^2}{u_1 u_2} \\
 & \ll & 2^{-3v} \left(c+\frac{1}{2}\right) \frac{(u_1,u_2)}{u_2},
\end{eqnarray*}
where the last estimate follows from the assumption of being in the case $2^{-v+1}\left (c+\frac{1}{2}\right) \frac{\pi u_1}{(u_1,u_2)} \geq 1$.\\

Now assume that we are in the case $2^{-v+1} \left(c+\frac{1}{2}\right) \frac{\pi u_1}{(u_1,u_2)} \leq 1 \leq 2^{-v+1}\left(c+\frac{1}{2} \right) \frac{\pi u_2}{(u_1,u_2)}$. Then, using $\sum_{r \geq x} r^{-2} \ll x^{-1}$, and writing $A = \left(2^{-v+1} \left(c+\frac{1}{2}\right) \frac{\pi u_1}{(u_1,u_2)}  \right)^{-1}$, the sum in lines \eqref{eq:sum_here1} and \eqref{eq:sum_here2} is bounded by
\begin{eqnarray*}
& & \sum_{r \geq 1} \left|\sin \left(2^{-v} \pi \frac{r u_1}{(u_1,u_2)} \right)  \sin \left(2^{-v} \pi \frac{r u_2}{(u_1,u_2)} \right)  \sin \left(2^{-v+1} \pi \left(c + \frac{1}{2} \right)  \frac{r u_1}{(u_1,u_2)} \right)  \right| \frac{(u_1,u_2)^4}{r^4 u_1^2 u_2^2} \\
 & \ll & \sum_{r \geq 1} 2^{-2v} \frac{r u_1}{(u_1,u_2)} \frac{r u_2}{(u_1,u_2)} \left| \sin \left(2^{-v+1} \pi \left(c + \frac{1}{2} \right)  \frac{r u_1}{(u_1,u_2)} \right)  \right| \frac{(u_1,u_2)^4}{r^4 u_1^2 u_2^2} \\
 & \ll &  \sum_{r < A} 2^{-3v} \left( c + \frac{1}{2} \right) \frac{(u_1,u_2)}{r u_2} +  \sum_{r \geq A} 2^{-2v} \frac{(u_1,u_2)^2}{r^2 u_1 u_2} \\
 & \ll &  2^{-3v} \left( c + \frac{1}{2} \right) \frac{(u_1,u_2)}{u_2} \left( 1 + \log  A \right)  +  A^{-1} 2^{-2v} \frac{(u_1,u_2)^2}{u_1 u_2}
 \\
 & \ll &  2^{-3v} (v+1) (c+1) \frac{(u_1,u_2)}{u_2},
\end{eqnarray*}
where we used that $\log A = \log \left(2^{-{v+1}} \left(c+\frac{1}{2} \right) \frac{\pi u_1}{(u_1,u_2)} \right)^{-1} \ll \log (2^{v}) \ll v$.\\

Finally, assume that we are in the case $2^{-v+1}\left (c+\frac{1}{2} \right) \frac{\pi u_2}{(u_1,u_2)} \leq 1$. Then, writing $A = \left(2^{-v+1} \left(c+\frac{1}{2} \right) \frac{\pi u_2}{(u_1,u_2)}\right)^{-1}$ and $B =  \left(2^{-v+1} \left(c+\frac{1}{2} \right) \frac{\pi u_1}{(u_1,u_2)}\right)^{-1}$, the sum in lines \eqref{eq:sum_here1} and \eqref{eq:sum_here2} is bounded by
\begin{eqnarray*}
 & \ll &\sum_{r \geq 1} 2^{-2v} \left| \sin \left(2^{-v+1}  \left(c + \frac{1}{2} \right)  \frac{\pi r u_1}{(u_1,u_2)} \right) \sin \left(2^{-v+1} \left(c + \frac{1}{2} \right)  \frac{\pi r u_2}{(u_1,u_2)} \right)  \right| \frac{(u_1,u_2)^2}{r^2 u_1 u_2} \\
 & \ll &  \sum_{1 \leq r < A} 2^{-4v} \left(c + \frac{1}{2} \right)^2 + \sum_{A \leq r < B} 2^{-3v} \left(c + \frac{1}{2} \right) \frac{(u_1,u_2)}{r u_2} +  \sum_{r \geq B} 2^{-2v}  \frac{(u_1,u_2)^2}{r^2 u_1 u_2} \\
 & \ll &  2^{-3v} \left(c + \frac{1}{2} \right) \frac{(u_1,u_2)}{u_2} + 2^{-3v} \left(c + \frac{1}{2} \right) (v+1)  \frac{(u_1,u_2)}{u_2} +  2^{-3v} \left(c+\frac{1}{2} \right) \frac{(u_1,u_2)}{u_2} \\
 & \ll & 2^{-3v} (v+1) \left(c+1 \right)  \frac{(u_1,u_2)}{u_2}.
\end{eqnarray*}
Thus, overall we have
\begin{eqnarray*}
\sum_{r  \geq 1} \left| \widehat{f_{v,c}}\left(\frac{ru_{2}}{\left(u_{1},u_{2}\right)}\right) \widehat{f_{v,c}}\left(\frac{ru_{1}}{\left(u_{1},u_{2}\right)}\right) \right| \ll  2^{-3v} (v+1) (c+1) \frac{(u_{1},u_{2})}{\max\{u_{1},u_{2}\} },
\end{eqnarray*}
as claimed. Accordingly, overall we have
\begin{eqnarray}
& & \int_{0}^{1}\left(\sum_{N_1 \leq n \leq N_2} Y_{n,v,c} (\alpha) \right)^{2}\,d\alpha \nonumber\\
& \ll & 2^{-3v}   (v+1) (c+1) \sum_{u_{1},u_{2}> 0}\textup{Rep}_{N_1,N_2}\left(u_{1}\right)\textup{Rep}_{N_1,N_2}\left(u_{2}\right)  \frac{(u_{1},u_{2})}{\max\{u_{1},u_{2}\} }, \label{eq_variance_upper_b}
\end{eqnarray}
which proves the lemma.
\end{proof}

We note that in the case $u_1 u_2 / (u_1,u_2)^2 \geq 2^{\frac{9v}{4}}$, we can also bound the expression in lines \eqref{eq:sum_here1} and \eqref{eq:sum_here2} more directly by
\begin{equation}
\sum_{r \geq 1} \frac{(u_1,u_2)^4}{r^4 u_1^2 u_2^2} \ll \frac{(u_1,u_2)^4}{u_1^2 u_2^2}  \ll 2^{-\frac{9v}{4}} \frac{(u_1,u_2)^2}{u_1 u_2}, \label{eq_in_combination}
\end{equation}
 which along the same lines as in the proof of \eqref{eq_variance_upper_b} leads to the (more complicated) estimate
\begin{eqnarray}
& & \int_{0}^{1}\left(\sum_{N_1 \leq n \leq N_2} Y_{n,v,c} (\alpha) \right)^{2}\,d\alpha   \nonumber\\
& \ll & 2^{-3v}  \sum_{u_{1},u_{2} > 0}\textup{Rep}_{N_1,N_2}\left(u_{1}\right)\textup{Rep}_{N_1,N_2} \left(u_{2}\right)  \times  \label{eq_variance_upper_b_2}\\
& &\qquad \times \left((v+1)(c+1) \frac{(u_{1},u_{2})}{\max\{u_{1},u_{2}\}}  \mathbf{1} \left( \frac{u_{1} u_{2}}{(u_{1},u_{2})^2} \leq 2^{\frac{9v}{4}} \right)  + 2^{\frac{3v}{4}} \frac{(u_1,u_2)^2}{u_1 u_2} \right). \nonumber
\end{eqnarray}
This will only be required for the quadratic case of Theorem \ref{thm:Main_thm}. For the case of degree $\geq 3$ of Theorem \ref{thm:Main_thm} we will use a version of \eqref{eq_variance_upper_b_2}, where the condition $u_1 u_2 / (u_1,u_2)^2 \geq 2^{3v}$ instead of $u_1 u_2 / (u_1,u_2)^2 \geq 2^{\frac{9v}{4}}$ is used. If we allow ourselves to lose a factor $N^\ve$ in the final result, as we do in the statement of Theorem \ref{thm:2}, then we can continue to estimate the quantities in Lemma \ref{prop:VarianceBound} as explained in the rest of this section. (In contrast, for Theorem \ref{thm:Main_thm} we need a more precise estimate, which only loses terms of logarithmic order. We will derive the necessary estimate in subsequent sections.)

\begin{lem} \label{prop:VarianceBound_coarse}
Let $\ve >0$ be fixed. Let $\textup{Rep}_{N_1,N_2} (u)$ be defined as in the statement of the previous lemma. Set
\begin{equation} \label{def_e}
E(N_1,N_2) := \sum_{u>0} \left(\textup{Rep}_{N_1,N_2} (u)\right)^2.
\end{equation}
Then for all $N_1$ and $N_2$ satisfying $1 \leq N_1 < N_2$, and for all $v,c$ in the range specified in \eqref{eq:param_ranges},  we have
$$
\mathbb{E}\left[\left(\sum_{N_1 \leq n \leq N_2} Y_{n,v,c} \right)^2 \right]  \ll_{\varepsilon} 2^{-3v} (v+1)(c+1)  N_2^{\varepsilon} E(N_1,N_2),
$$
with an implied constant that is independent of $N_1,N_2,v,c$.
\end{lem}

\begin{proof}
By Lemma \ref{prop:VarianceBound} we have
\begin{eqnarray}
 && \mathbb{E}\left[\left(\sum_{N_1 \leq n \leq N_2} Y_{n,v,c} \right)^2 \right] \nonumber\\
 & \ll & 2^{-3v}   (v+1)\left( c+1 \right)  \sum_{u_{1},u_{2} > 0}\textup{Rep}_{N_1,N_2}\left(u_{1}\right)\textup{Rep}_{N_1,N_2}\left(u_{2}\right) \frac{(u_{1},u_{2})}{\max\{u_{1},u_{2}\} } \nonumber\\
 & \ll & 2^{-3v} (v+1)\left( c+1 \right)  \sum_{u_{1},u_{2} > 0}\textup{Rep}_{N_1,N_2}\left(u_{1}\right)\textup{Rep}_{N_1,N_2}\left(u_{2}\right) \frac{(u_{1} , u_{2})}{\sqrt{u_{1} u_{2}}}. \label{eq_gcd_arises}
\end{eqnarray}
The sum in the last line of the previous equation is called a \emph{GCD sum} (with parameter 1/2, corresponding to the square-root in the denominator of the final term). Such sums are known to play a key role in metric number theory, and have been intensively studied (see for example Chapter 3 in Harman's book \cite{harman}, the classical papers \cite{dyer,gal,koks}, as well as the recent papers \cite{abs,bonds,lr}). Optimal upper bound for such sums with parameter 1/2 were finally established by de la Bret\`eche and Tenenbaum \cite{dlbt}. For our purpose the estimate
$$
 \sum_{u_{1},u_{2} > 0}\textup{Rep}_{N_1,N_2}\left(u_{1}\right)\textup{Rep}_{N_1,N_2}\left(u_{2}\right)  \frac{(u_{1} , u_{2})}{\sqrt{u_{1} u_{2}}}  \ll_\ve N_2^{\ve} \sum_{u > 0} \left(\textup{Rep}_{N_1,N_2}\left(u\right) \right)^2 = N_2^{\ve} E(N_1,N_2)
$$
is sufficient (which follows for example from \cite[Theorem 1.2]{dlbt}). From this and \eqref{eq_gcd_arises} we obtain
$$
\mathbb{E}\left[\left(\sum_{N_1 \leq n \leq N_2} Y_{n,v,c} \right)^2 \right]  \ll_\ve 2^{-3v} (v+1)(c+1) N_2^{\ve} E(N_1,N_2),
$$
which proves the lemma.
\end{proof}

It can be easily checked that the function $E(N_1,N_2)$ satisfies all the requirements which are made to the function $G$ in the statement of Lemma \ref{lemma_mss}.  This essentially follows from $\textup{Rep}_{N_1,N_3}\left(u\right) = \textup{Rep}_{N_1,N_2}\left(u\right) + \textup{Rep}_{N_2+1,N_3}\left(u\right)$ for all $u$, so that
\begin{eqnarray*}
& & \textup{Rep}_{N_1,N_3}\left(u_{1}\right)\textup{Rep}_{N_1,N_3} \left(u_{2}\right) \\
& = & \left( \textup{Rep}_{N_1,N_2}\left(u_1\right) + \textup{Rep}_{N_2+1,N_3}\left(u_1\right) \right) \left( \textup{Rep}_{N_1,N_2}\left(u_2\right) + \textup{Rep}_{N_2+1,N_3}\left(u_2\right) \right) \\
& \geq & \textup{Rep}_{N_1,N_2}\left(u_1\right) \textup{Rep}_{N_1,N_2}\left(u_2\right) + \textup{Rep}_{N_2+1,N_3}\left(u_1\right)\textup{Rep}_{N_2+1,N_3}\left(u_2\right) .
\end{eqnarray*}
Thus as a consequence of Lemma \ref{lemma_mss} we obtain

\begin{lem} \label{prop:VarianceBound_coarse_2}
Let $\ve>0$ be fixed. Then
$$
 \mathbb{E}\left[ \left( \max_{1 \leq M \leq N} \left| \sum_{n=1}^M Y_{n,v,c} \right| \right)^2 \right] \ll_\ve 2^{-3v} (v+1)(c+1)  N^{\varepsilon} E(1,N),
 $$
 with implied constants independent of $v,c,N$.
\end{lem}

\section{Proof of Theorem \ref{thm:2}} \label{sec_proof_thm}

We come to the proof of Theorem \ref{thm:2}. The proof of Theorem \ref{thm:Main_thm} will be based on a very similar line of reasoning, but requires a more careful analysis since we there allow ourselves only error terms of logarithmic order (in comparison to $N^\ve$ errors which are admissible in the proof of Theorem \ref{thm:2}).\\

Let $\eta$ be the constant from the statement of Theorem \ref{thm:2}. Throughout the proof, let $\ve>0$ be fixed. Letting $\textup{Rep}$ and $E$ be defined as in \eqref{def_rep} resp.\ \eqref{def_e}, the assumption of Theorem \ref{thm:2} can be written as
$$
E(1,N) \ll_\eta N^{2 + \eta}.
$$
Thus, as a consequence of Lemma \ref{prop:VarianceBound_coarse_2}, we have
\begin{equation} \label{eq_var_bound_eta}
 \mathbb{E}\left[ \left( \max_{1 \leq M \leq N} \left| \sum_{n=1}^M Y_{n,v,c} \right| \right)^2 \right] \ll_{\ve,\eta} 2^{-3v} (v+1)(c+1)  N^{2 + \eta + \varepsilon/2},
\end{equation}
 with implied constants independent of $v,c,N$.

\begin{lem}
\label{lem:Good_Alphas}For almost all $\alpha\in\left[0,1\right]$,
there exists $N_{0}=N_{0}\left(\alpha\right)$ such that for all $N\ge N_{0}$
we have
\[
\left| \sum_{n=1}^N Y_{n,v,c}\right|\leq 2^{-v} v^{-2}(c+1) N,
\]
for all $v,c$ in the ranges  $ v \ge (\eta +  \ve) \log_2 N  ,\, 0 \leq c < 2^{v}$.
\end{lem}

\begin{proof}
Let
\begin{equation} \label{eq_A_def}
\mathcal{A}_{m,v,c} :=\left\{ \alpha\in\left[0,1\right]:\,\max_{2^{m} < M \leq 2^{m+1}} \left| \sum_{n=1}^M Y_{n,v,c}\right| > 2^{-v} v^{-2} (c+1) 2^m \right\}.
\end{equation}
Then by \eqref{eq_var_bound_eta} and by Chebyshev's inequality we have
\[
\text{meas}\left(\mathcal{A}_{m,v,c}\right)\ll 2^{-v} v^5(c+1)^{-1} 2^{m \eta + m \ve/2}.
\]
Thus we have
\begin{align*}
\sum_{\substack{v \ge \eta m + \ve m,\\0 \leq c < 2^{v}}  } \text{meas}\left( \mathcal{A}_{m,v,c}\right)  & \ll \sum_{\substack{v \ge \eta m + \ve m,\\0 \leq c < 2^{v}} } 2^{-v}v^5 (c+1)^{-1} 2^{m \eta + m \ve/2} \\
& \ll \sum_{v \geq \eta m + \ve m} v^6 2^{-v}  2^{m \eta + m \ve/2} \\
& \ll m^6 2^{-m\ve/2},
\end{align*}
so that
$$
\sum_{m=1}^\infty ~ \sum_{\substack{v \ge \eta m + \ve m ,\\0 \leq c < 2^{v}}} \text{meas}\left( \mathcal{A}_{m,v,c}\right) < \infty.
$$
Accordingly, by the Borel-Cantelli lemma with full probability only finitely many events $\mathcal{A}_{m,v,c}$ occur. Since every integer $N$ falls into a suitable dyadic range $(2^{m},2^{m+1}]$, we obtain the desired conclusion.
\end{proof}

We now proceed with the proof of Theorem \ref{thm:2}. Let $S$ be given, and assume that $S \in [0, N^{-\eta-\ve}]$. Clearly $S$ can be written in binary representation as
$$
S = \sum_{v=0}^\infty d_v 2^{-v}
$$
for suitable digits $d_v \in\{0,1\}$.
%Let $S^-$ be defined as $S^- = \sum_{v=0}^{\log_2 N + \log \log N } d_v 2^{-v}$. Note that by construction $S- S^- = o\left(N^{-1} \right)$.
%Clearly
%$$
%\psi_{S/2}(x) \geq \psi_{S^-/2} (x), \qquad x \in \mathbb{R},
%$$
%and
%$$
%\int_\mathbb{R} \psi_{S/2}(x) dx - \int_\mathbb{R} \psi_{S^-/2}(x) dx = S^2  -  \left(S^- \right)^2 = o \left(N^{-1}\right).
%$$
We have
\[ \psi_{S/2}(x)=  \sum_{v=0}^\infty d_v f_{v,c_v}(x) \]
for suitable coefficients $c_v$ which arise from the dyadic decomposition of $S$ via
\begin{equation} \label{eq_dyad_arise}
c_v = \sum_{u=0}^{v-1} 2^{v-u} d_u,
\end{equation}
uniformly in $ x\in \mathbb{R}. $ We thus have
\begin{eqnarray*}
& & \sum_{m \neq n =1}^N \sum_{j \in \mathbb{Z}} \psi_{S/2} (\alpha x_n - \alpha x_m + j) \\
& = & \sum_{n=1}^N \sum_{v=0}^{\infty} d_v \left( Y_{n,v,c_v} + 2 (n-1)(2c_v+1) 2^{-2v} \right).
\end{eqnarray*}

By symmetry, we have the following identity:

\begin{eqnarray*}
S^2 &=& \left( \sum_{v=0}^{\infty} d_v 2^{-v} \right)^2 = \sum_{v=0}^\infty \sum_{u=0}^\infty d_u d_v 2^{-u-v} \\
&=& 2\sum_{v=0}^\infty \sum_{u=0}^{v-1} d_u d_v 2^{-u-v} + \sum_{v=0}^{\infty} d_v 2^{-2v} \\
&=&  \sum_{v=0}^{\infty} d_v \left(2\sum_{u=0}^{v-1} 2^{v-u} d_u +1 \right) 2^{-2v}.
\end{eqnarray*}

Thus,
\begin{eqnarray*}
& &  \sum_{n=1}^N \sum_{v=0}^{\infty} d_v 2 (n-1)(2c_v+1) 2^{-2v}  \\
& = & (N^2 - N) \sum_{v=0}^{\infty} d_v (2c_v+1) 2^{-2v} \\
& = &  (N^2 - N) \sum_{v=0}^{\infty} d_v \left(2\sum_{u=0}^{v-1} 2^{v-u} d_u +1 \right) 2^{-2v} \\
& = & (N^2 - N) S^2 \\
& = & N^2 S^2 + o (N S).
\end{eqnarray*}

Note that the assumption on the size of $S$ implies that $d_v = 0$ for $v < (\eta + \ve) \log_2 N$. Hence, we can restrict to $ v \ge   (\eta + \ve) \log_2 N $, so we are within the admissible range of Lemma \ref{lem:Good_Alphas}.
Moreover, we have
$$
c_v = 2^v \sum_{u=0}^{v-1} 2^{-u} d_u \leq 2^v S,
$$
and if $ d_v \ne 0 $ then $ S \ge 2^{-v} $, and therefore $(c_v+1)2^{-v} \ll S$.
Assuming that $\alpha$ is from the generic set in Lemma \ref{lem:Good_Alphas} and that $N$ is sufficiently large, it follows that
\begin{eqnarray}
\left| \sum_{n=1}^N \sum_{v=0}^{
\infty} d_v Y_{n,c_v,v} \right| & \le &   \sum_{v \ge (\eta+\ve)\log_2 N} d_v 2^{-v} v^{-2} (c_v+1) N  \label{eq_split_v} \\
& \ll & NS  \sum_{v \ge (\eta+\ve)\log_2 N} v^{-2}\\
& \ll & \frac{NS}{\log N} = o(NS). \nonumber
\end{eqnarray}
Thus
\begin{equation} \label{eq_ns_error}
 \left| \sum_{m \neq n =1}^N \sum_{j \in \mathbb{Z}} \psi_{S/2} (\alpha x_n - \alpha x_m + j) - N^2 S^2 \right| = o(NS)
\end{equation}
%An analogous argument can be carried out for $S^+ = \sum_{v=0}^{ \log_2 N + \log \log N} d_v 2^{-v} + 2^{-{\lfloor \log_2 N + \log \log N \rfloor}}.$ Then by construction $S^+ \geq S^-$ and $S^+ - S = o \left( N^{-1} \right)$, and we can establish a perfect analogue of \eqref{eq_ns_error} with $S^+$ in place of $S^-$. Since $\psi_{S^-/2} \leq \psi_{S/2} \leq \psi_{S^+/2}$ holds pointwise, overall we obtain
%$$
% \left| \sum_{m \neq n =1}^N \sum_{j \in \mathbb{Z}} \psi_{S/2} (\alpha x_n - \alpha x_m + j) - N^2 S^2 \right| = o(NS)
%$$
as $N \to \infty$, uniformly for all $S \in [0, N^{-\eta-\ve}]$, provided that $\alpha$ is from the generic set in Lemma \ref{lem:Good_Alphas}.  This establishes \eqref{eq_num_var} uniformly throughout the desired range of $S$ and proves Theorem \ref{thm:2}.

\section{Divisors for polynomial sequences: the quadratic case} \label{sec_poly_deg2}

In this section we give the proof of Theorem \ref{thm:Main_thm} in the case of degree $2$. The case of degree $d \geq 3$ will be treated in the following section. Throughout this section, let $p(x) \in \mathbb{Z}[x]$ be a fixed quadratic polynomial, and let $x_n = p(n)$ for $n \geq 1$. The key ingredient is again to control the variance with respect to $\alpha$. As the previous section indicated, the size of the variance depends on an interplay of the representation numbers $\textup{Rep}_{1,N}(u_1)$ and $\textup{Rep}_{1,N}(u_2)$ on the one hand, and the size of greatest commons divisors $(u_1,u_2)$ on the other hand. In the case of sequences of polynomial origin, in principle we have good control of both effects (for $u_1$ and $u_2$ that are represented as differences $p(m)-p(n)$). However, the problem is that we need to control both effects simultaneously; while we know that there are only few $u_1$ and $u_2$ for which the representation numbers $\textup{Rep}_{1,N}(u_1)$ and $\textup{Rep}_{1,N}(u_2)$ are large, and very few $u_1$ and $u_2$ that have a particularly large gcd, we need to rule out the potential ``conspiracy'' that those (very unusual) $u_1$ and $u_2$ that have particularly large representation numbers are exactly the same (very unusual) $u_1$ and $u_2$ that have a particularly large gcd.\\

Let $Y_{n,v,c}$ be defined as in the previous section. We prove:

\begin{lem} \label{prop:VarianceBound_deg2}
For all $v$ and $c$ in the range \eqref{eq:param_ranges} we have
$$
 \mathbb{E}\left[ \left( \max_{1 \leq M \leq N} \left| \sum_{n=1}^M Y_{n,v,c} \right| \right)^2 \right] \ll 2^{-\frac{9v}{4}} N^2 (\log N)^{7} (v+1) (v + \log \log N)(c+1)  ,
 $$ with implied constants independent of $v,c,N$.
\end{lem}

Before giving the proof of Lemma \ref{prop:VarianceBound_deg2} we collect some necessary ingredients. The first gives a bound for the number of the representations of integers as differences of quadratic polynomials.

\begin{lem} \label{lemma_blogra}
Let $p(x) \in \mathbb{Z}[x]$ be a polynomial of degree $2$. Let a positive integer $\beta$ be fixed. Let
$$
r (u) = \# \left\{n_1, n_2:~p(n_1) - p(n_2) = u \right\}.
$$
Then
$$
\sum_{1 \leq u \leq x} \left(r (u) \right)^\beta \ll_{\beta} x (\log x)^{2^{\beta}-1} \qquad \text{as $x \to \infty$.}
$$
\end{lem}

\begin{proof} Problems of this type have been studied in great details and in a very general setup; see for example \cite[Theorem 3]{blogra} (which is not directly applicable here, because the quadratic form in that paper is assumed to be positive). Since we only need an upper bound (and not a precise asymptotics), the desired result can be obtained very quickly. Let $p(x) = a x^2 + bx + c$, and assume that for some $u > 0$ we have
\begin{equation} \label{eq_u_fac}
u = p(x) - p(y) = a x^2 + bx - a y^2 - by = \left(a (x+y)+b)\right) (x-y).
\end{equation}
Thus the difference $x-y$ must be a divisor of $u$, and it is easy to see that $x$ and $y$ are uniquely determined in the factorization \eqref{eq_u_fac} by the value of $x-y$. Thus we have $r(u) \leq \tau(u)$, where $\tau$ is the number-of-divisors function.   The function $\tau$ is multiplicative; the asymptotics of its moments follows for example from an application of Wirsing's theorem (see e.g.\ Chapter~2 of \cite{schwa}), but can also be calculated using elementary methods \cite{luca}. One has
$$
\sum_{n \leq x} \tau(n)^\beta \sim c_\beta x (\log x)^{2^{\beta}-1},
$$
which yields the desired upper bound.
\end{proof}

The next lemma concerns the average order of greatest common divisors. This is stated for example as Theorem 4.3 in \cite{brough} and as Equation (17) in \cite{toth}.

\begin{lem} \label{lemma_toth}
$$
\sum_{n\le x} \sum_{1 \leq m < n} \frac{(m,n)}{n} = \frac{6}{\pi^2} x \log x + \mathcal{O} (x).
$$
\end{lem}

Now we are in a position to give a proof for the variance estimate in Lemma \ref{prop:VarianceBound_deg2}. We acknowledge the fact that the sequence $(x_n)_{n =1}^\infty = (p(n))_{n =1}^\infty$ in this (and the next) section is in general not necessarily a sequence of distinct positive integers (even if we assume w.l.o.g.\ that the leading coefficient is positive and not negative, which is possible by simply replacing $p$ by $-p$ if necessary), since the same value can occur in the sequence multiple times. However, this only affects finitely many elements at the initial segment of the sequence, and does not affect the asymptotic behavior, so to keep the notation as simple as possible we ignore this fact and assume w.l.o.g.\ that $(p(n))_{n=1}^\infty$ is indeed a sequence of distinct positive integers.

\begin{proof}[Proof of Lemma \ref{prop:VarianceBound_deg2}]
Let $\Psi \geq 1$ be a positive integer (to be chosen later). The proof will be similar to that of Lemma \ref{prop:VarianceBound}, but instead of going directly from \eqref{eq_var_1} to \eqref{eq_ingred_2}, we first split $\mathbb{N}$ into $\Psi$ many classes
$$
A_w := \left\{ u \in \mathbb{N}:~ 2^{w-1} \leq \textup{Rep}_{1,N}\left(u\right) < 2^w \right\}, \qquad 1 \leq w \leq \Psi-1,
$$
and
$$
A_{\Psi} := \left\{ u \in \mathbb{N}:~ \textup{Rep}_{1,N}\left(u\right) \geq 2^{\Psi-1} \right\}.
$$
Then we apply Cauchy-Schwarz to get
\begin{eqnarray}
&& \int_{0}^{1}\left( \max_{1 \leq M \leq N}  \left| \sum_{1 \leq n \leq M} Y_{n,v,c} (\alpha) \right| \right)^{2}\,d\alpha \nonumber\\
& = &  \int_0^1 \left( \max_{1 \leq M \leq N} \left|  2 \sum_{u>0} \textup{Rep}_{1,M}\left(u\right) \sum_{j\ne 0} \widehat{f_{v,c}}\left(j\right) e^{2\pi ij\alpha u} \right| \right)^2 d\alpha  \nonumber  \\
& \leq & 4 \Psi    \sum_{w=1}^\Psi \int_0^1 \left( \max_{1 \leq M \leq N} \left|  \sum_{u \in A_w} \textup{Rep}_{1,M}\left(u\right) \sum_{j\ne 0} \widehat{f_{v,c}}\left(j\right) e^{2\pi ij\alpha u} \right| \right)^2 d\alpha. \label{eq_cau_sch}
\end{eqnarray}
Proceeding as in the previous section, as an analogue of \eqref{eq_variance_upper_b_2} we obtain
\begin{eqnarray*}
& & \int_0^1 \left(  \sum_{u \in A_w}  \textup{Rep}_{N_1,N_2}\left(u\right)  \sum_{j\ne 0} \widehat{f_{v,c}}\left(j\right) e^{2\pi ij\alpha u} \right)^2 d\alpha \\
& \ll & 2^{-3v} \sum_{u_{1},u_{2} \in A_w}\textup{Rep}_{N_1,N_2}\left(u_{1}\right)\textup{Rep}_{N_1,N_2} \left(u_{2}\right)  \times   \nonumber\\
& &\qquad \times \left((v+1)(c+1)  \frac{(u_{1},u_{2})}{\max\{u_{1},u_{2}\}}  \mathbf{1} \left( \frac{u_{1} u_{2}}{(u_{1},u_{2})^2} \leq 2^{\frac{9v}{4}} \right)  + 2^{\frac{3v}{4}} \frac{(u_1,u_2)^2}{u_1 u_2} \right). \nonumber
\end{eqnarray*}
Note again that the right-hand side of this equation, interpreted as a function in the two variables $N_1$ and $N_2$, satisfies the requirements that were made to the function $G$ in the statement of Lemma \ref{lemma_mss}. Accordingly, by Lemma \ref{lemma_mss} we have
\begin{eqnarray}
& & \int_0^1 \left( \max_{1 \leq M \leq N} \left| \sum_{u \in A_w}  \textup{Rep}_{1,M}\left(u\right)  \sum_{j\ne 0} \widehat{f_{v,c}}\left(j\right) e^{2\pi ij\alpha u} \right|^2  \right) d\alpha  \nonumber\\
& \ll & 2^{-3v}  (\log N)^2 \sum_{u_{1},u_{2} \in A_w}\textup{Rep}_{1,N}\left(u_{1}\right)\textup{Rep}_{1,N} \left(u_{2}\right)  \times   \nonumber\\
& &\qquad \times \left(  (v+1) (c+1) \frac{(u_{1},u_{2})}{\max\{u_{1},u_{2}\}}   \mathbf{1} \left( \frac{u_{1} u_{2}}{(u_{1},u_{2})^2} \leq 2^{\frac{9v }{4}} \right)  + 2^{\frac{3v}{4}} \frac{(u_1,u_2)^2}{u_1 u_2} \right). \label{eq_max_var}
\end{eqnarray}
We note that in this and the following section the application of Lemma \ref{lemma_mss} (rather than obtaining the maximal inequality ``by hand'', using a  dyadic decomposition method) is a very convenient auxiliary means, since the ``local'' representation numbers $\textup{Rep}_{N_1,N_2}\left(u\right)$ are very difficult to control when $N_2-N_1$ is small, while strong estimates for the ``global'' representation numbers $\textup{Rep}_{1,N}\left(u\right)$ (such as in our Lemma \ref{lemma_blogra}) are available from the literature.\\

The sum
$$
\sum_{u_{1},u_{2} \in A_w} \textup{Rep}_{1,N}\left(u_{1}\right)\textup{Rep}_{1,N} \left(u_{2}\right)   \frac{(u_1,u_2)^2}{u_1 u_2}
$$
is a GCD sum with parameter 1, and from the literature (e.g.\ \cite[Theorem 2]{lr}) we get
\begin{equation} \label{eq_gcd_1}
\sum_{u_{1},u_{2} \in A_w} \textup{Rep}_{1,N}\left(u_{1}\right)\textup{Rep}_{1,N} \left(u_{2}\right)   \frac{(u_1,u_2)^2}{u_1 u_2}  \ll (\log \log N)^2 \sum_{u \in A_w} \textup{Rep}_{1,N}\left(u\right)^2.
\end{equation}

Thus the main task is to estimate
$$
 \sum_{u_{1},u_{2} \in A_w} \textup{Rep}_{1,N}\left(u_{1}\right)\textup{Rep}_{1,N} \left(u_{2}\right) \frac{(u_{1},u_{2})}{\max\{u_{1},u_{2}\}}  \mathbf{1} \left( \frac{u_{1} u_{2}}{(u_{1},u_{2})^2} \leq 2^{\frac{9v}{4}} \right).
$$
Trivially for all $N$ and all $u$ we have $\textup{Rep}_{1,N}\left(u\right) \leq r(u)$ for all $u$, where $r(u)$ is the function from the statement of Lemma \ref{lemma_blogra}. Note furthermore that $\textup{Rep}_{1,N}\left(u\right) \neq 0$ is only possible for $u \ll N^2$.
On the one hand, by Lemma \ref{lemma_toth}, for $u_1,u_2 \in A_w$ and $1 \leq w \leq \Psi-1$ we have
\begin{eqnarray}
\sum_{u_{1},u_{2} \in A_w} \textup{Rep}_{1,N}\left(u_{1}\right)\textup{Rep}_{1,N} \left(u_{2}\right) \frac{(u_{1},u_{2})}{\max\{u_{1},u_{2}\}} & \leq & 2^{2w} \sum_{u_{1},u_{2} \ll N^2} \frac{(u_{1},u_{2})}{\max\{u_{1},u_{2}\}} \nonumber\\
& \ll & 2^{2w} N^2 \log N.\label{eq_aw_sum}
\end{eqnarray}
On the other hand, using Lemma \ref{lemma_blogra} with $\beta =3$, we have
$$
\sum_{u>0} (\textup{Rep}_{1,N}\left(u\right))^3 \ll N^2 (\log N)^{7},
$$
which implies that
\begin{equation} \label{eq_aw_size}
\# A_w \ll \frac{N^2 (\log N)^{7}}{2^{3w}}.
\end{equation}
Assume w.l.o.g.\ that $u_1 \leq u_2$. Then for every fixed $u_2$ and for every $\ell = 1,2,\dots$ there are at most $\ell$ (more precisely, $ \phi(\ell)$, where $ \phi $ is Euler's totient function) many different integers $u_1$ of size $ \leq u_2$ such that
\begin{equation} \label{eq_size_gcd}
(u_{1},u_{2}) = \frac{u_2}{\ell},
\end{equation}
namely the integers
\begin{equation} \label{eq_size_gcd_2}
u_1 = \frac{k u_2}{\ell}, \qquad k \leq \ell, \, (k,\ell) = 1.
\end{equation}
(assuming that $u_2$ is actually divisible by $\ell$, otherwise there are no such $u_1$ at all).  We note that \eqref{eq_size_gcd} and \eqref{eq_size_gcd_2} imply that
$$
 \frac{u_{1} u_{2}}{(u_{1},u_{2})^2} = \frac{\frac{k u_2}{\ell} u_2}{\left( \frac{u_2}{\ell} \right)^2 } = k \ell,
$$
so that the condition
$$
\frac{u_{1} u_{2}}{(u_{1},u_{2})^2} \leq 2^{\frac{9v}{4}}
$$
implies that
$$
k \ell \leq 2^{\frac{9v }{4}}.
$$
As a consequence, for every $u_2$ we have
\begin{eqnarray}
\sum_{1 \leq u_1 \leq u_2} \frac{(u_{1},u_{2})}{\max\{u_{1},u_{2}\}} \mathbf{1} \left( \frac{u_{1} u_{2}}{(u_{1},u_{2})^2} \leq 2^{\frac{9v}{4}} \right) & \leq & \sum_{\substack{1 \leq k \leq \ell,\\ k \ell \leq 2^{\frac{9v}{4}}}} \frac{1}{\ell} \nonumber\\
& \ll & \sum_{\ell \leq 2^{\frac{9v }{4}}} \min \left\{ \ell,  2^{\frac{9v}{4}} \ell^{-1} \right\} \frac{1}{\ell}   \nonumber\\
& \ll &  2^{\frac{9v}{8}}. \label{eq_u1_sum}
\end{eqnarray}
Accordingly, in addition to \eqref{eq_aw_sum}, for $1 \leq w \leq \Psi-1$ we have
\begin{eqnarray}
& &  \sum_{u_{1},u_{2} \in A_w} \textup{Rep}_{1,N}\left(u_{1}\right)\textup{Rep}_{1,N} \left(u_{2}\right) \frac{(u_{1},u_{2})}{\max\{u_{1},u_{2}\}} \mathbf{1} \left( \frac{u_{1} u_{2}}{(u_{1},u_{2})^2} \leq 2^{\frac{9v }{4}} \right) \nonumber\\
& \leq & 2^{2w}  2^{\frac{9v}{8}}  \# A_w \nonumber\\
& \ll & \frac{2^{\frac{9v}{8}} N^2 (\log N)^{7}}{2^{w}}, \label{eq_u1u2_aw}
\end{eqnarray}
where we used \eqref{eq_aw_size} to estimate $\# A_w$.  By combining \eqref{eq_aw_sum} and \eqref{eq_u1u2_aw}, for $1 \leq w \leq \Psi-1$ we have
\begin{eqnarray}
 & & \sum_{u_{1},u_{2} \in A_w} \textup{Rep}_{1,N}\left(u_{1}\right)\textup{Rep}_{1,N} \left(u_{2}\right) \frac{(u_{1},u_{2})}{\max\{u_{1},u_{2}\}} \mathbf{1} \left( \frac{u_{1} u_{2}}{(u_{1},u_{2})^2} \leq 2^{\frac{9v }{4}} \right)  \nonumber\\
 & \ll & N^2 \log  N \min\left\{ 2^{2w}, 2^{\frac{9v}{8}} 2^{-w} (\log N)^{6} \right\} \label{eq_min_up}.
 \end{eqnarray}
For $A_\Psi$, by \eqref{eq_size_gcd_2}  we have
\begin{eqnarray}
& & \sum_{u_{1},u_{2} \in A_\Psi} \textup{Rep}_{1,N}\left(u_{1}\right)\textup{Rep}_{1,N} \left(u_{2}\right) \frac{(u_{1},u_{2})}{\max\{u_{1},u_{2}\}}\mathbf{1} \left( \frac{u_{1} u_{2}}{(u_{1},u_{2})^2} \leq 2^{\frac{9v }{4}} \right)  \nonumber\\
& \ll & \sum_{u \in A_\Psi} \textup{Rep}_{1,N}\left(u\right) \sum_{\substack{1 \leq k \leq \ell \leq 2^{\frac{9v}{4}}, \\(k,\ell)=1,~\ell \mid u}} \frac{1}{\ell}\, \textup{Rep}_{1,N} \left( \frac{k u}{\ell} \right)  \nonumber\\
& \ll &  \sum_{\substack{1 \leq k \leq \ell \leq 2^{\frac{9v}{4}}, \\(k,\ell)=1}} \frac{1}{\ell} \sum_{u \in A_\Psi} \left( \textup{Rep}_{1,N}\left(u\right) \right)^2 \nonumber\\
& \ll & 2^{\frac{9v}{4}} \sum_{u \in A_\Psi} \left(\textup{Rep}_{1,N} \left(u\right)\right)^2 \nonumber\\
& \ll &  2^{\frac{9v}{4}}  2^{-\Psi} \sum_{u \in A_\Psi} \left(\textup{Rep}_{1,N} \left(u\right)\right)^3 \nonumber\\
& \ll & 2^{\frac{9v}{4}} 2^{-\Psi} N^2 (\log N)^7 \nonumber\\
& \ll & N^2 \label{eq_upper_Psi}
\end{eqnarray}
after using Lemma \ref{lemma_blogra} with $\beta =3$, and choosing $\Psi = \log_2 (2^{\frac{9v}{4}} (\log N)^7) \ll v  + \log \log N$. With this choice of $\Psi$, using \eqref{eq_cau_sch} together with \eqref{eq_max_var}, \eqref{eq_gcd_1}, \eqref{eq_min_up}, \eqref{eq_upper_Psi} and summing over $w$, we obtain
\begin{eqnarray*}
& & \int_{0}^{1}\left( \max_{1 \leq M \leq N}  \left| \sum_{1 \leq n \leq M} Y_{n,v,c} (\alpha) \right| \right)^{2}\,d\alpha  \\
& \ll & 2^{-3v} (\log N)^2 \Psi \Bigg(  \sum_{w=1} ^{\Psi}   \Bigg( N^2 (\log N) (v+1)(c+1)  \min\left\{ 2^{2w}, 2^{\frac{9v}{8}} 2^{-w} (\log N)^{6} \right\} \\
& & \qquad \qquad \qquad \qquad \qquad+ (\log \log N)^2 2^{\frac{3v}{4}} \sum_{u \in A_w} \textup{Rep}_{1,N} (u)^2   \Bigg)  + (v+1)(c+1)N^2 \Bigg).
\end {eqnarray*}
%It will be seen that the very last term in the previous equation, i.e. the ``$+N^2$'' coming from \eqref{eq_upper_Psi}, is negligible in comparison with the other terms.
As a consequence of Lemma \ref{lemma_blogra} we have
$$
 \sum_{w=1}^\Psi \sum_{u \in A_w} \textup{Rep}_{1,N} (u)^2 \ll \sum_{u \ll N^2} r(u)^2 \ll N^2 (\log N)^3.
$$
Finally, it is easily checked that
$$
\sum_{w=1}^{\infty} \min\left\{ 2^{2w}, 2^{\frac{9v}{8}} 2^{-w} (\log N)^{6} \right\} \ll 2^{\frac{3v}{4}} (\log N)^{4}.
$$
Overall, this yields
$$
 \int_{0}^{1}\left( \max_{1 \leq M \leq N}  \left| \sum_{1 \leq n \leq M} Y_{n,v,c} (\alpha) \right| \right)^{2}\,d\alpha \ll 2^{-\frac{9v}{4}} N^2 (\log N)^{7} (v+1)(v + \log \log N) (c+1) ,
$$
as claimed.
\end{proof}

Note that a direct application of Lemma \ref{prop:VarianceBound_coarse_2}, together with the fact that for a sequence $(x_n)_{n=1}^\infty$ arising from a quadratic polynomial we have $E(1,N) \ll N^2 \log N$ (as a consequence of Lemma \ref{lemma_blogra}), would yield
$$
 \int_{0}^{1}\left( \max_{1 \leq M \leq N}  \left| \sum_{1 \leq n \leq M} Y_{n,v,c} (\alpha) \right| \right)^{2}\,d\alpha  \ll_\ve 2^{-3v} (v+1)(c+1) N^{2 + \ve}
$$
for $\ve>0$. When comparing this to the conclusion of Lemma \ref{prop:VarianceBound_deg2}, we see that we have traded a better rate in $N$ against a worse rate in $2^{-v}$. It is suitable to think of $2^{-v}$ as representing the length of the test function, so that essentially $S = 2^{-v}$. Thus Lemma \ref{prop:VarianceBound_deg2} gives a benefit over Lemma  \ref{prop:VarianceBound_coarse_2} when dealing with test functions that correspond to large intervals (i.e., large values of $S$), which is exactly what we try to achieve in Theorem \ref{thm:Main_thm}. The variance estimate from Lemma \ref{prop:VarianceBound_deg2} tells us that, heuristically, we should expect to see fluctuations of size roughly of the order of the square-root of the variance, i.e. of order
$$
2^{\frac{-9v}{8}} N (\log N)^{\frac{7}{2}}.
$$
We want this to be of order $o(SN) = o(2^{-v} N)$, which will indeed be the case if $S = 2^{-v}$ is at least as large as $(\log N)^c$ for a sufficiently large value of $c$.

\begin{proof}[Proof of Theorem \ref{thm:Main_thm} in the case of degree 2]
In principle the proof works along similar lines as the proof of Theorem \ref{thm:2}. Instead of \eqref{eq_A_def}, we now define sets
$$
\mathcal{A}_{m,v,c}=\left\{ \alpha\in\left[0,1\right]:\,\max_{2^{m} < M \leq 2^{m+1}} \left| \sum_{n=1}^M Y_{n,v,c}\right| > 2^{-v} v^{-2} (c+1) 2^{m} \right\},
$$
for the wider range
\begin{equation} \label{eq_v_range}
 v \ge (32+\ve) \log_2 m, \qquad 0 \leq c < 2^v.
\end{equation}
Then, by Lemma \ref{prop:VarianceBound_deg2} and Chebyshev's inequality, we have
\begin{eqnarray*}
\textup{meas} \left( A_{m,v,c} \right)  & \ll & 2^{-\frac{9v}{4}} 2^{2m} m^{7}(v+1)  (v + \log m) (c+1) 2^{2v} v^4 (c+1)^{-2} 2^{-2m} \\
& \ll & 2^{-\frac{v}{4}} v^6 m^{7} (\log m) (c+1)^{-1}.
\end{eqnarray*}
Accordingly,
\begin{eqnarray*}
& & \sum_{m \geq 1} \,  \sum_{v \ge (32+\ve) \log_2 m } \, \sum_{0 \leq c< 2^v } \textup{meas} \left(  \mathcal{A}_{m,v,c}  \right)  \\
& \ll &  \sum_{m \geq 1} \, \sum_{v \geq {(32+\ve)} \log_2 m} \, \sum_{0 \leq c < 2^v}  2^{-\frac{v}{4}}  v^6 m^{7} (\log m) (c+1)^{-1}  \\
& \ll &    \sum_{m \geq 1} \, \sum_{v \geq {(32+\ve)} \log_2 m} 2^{-\frac{v}{4}}  v^7  m^{7} (\log m) \\
& \ll &  \sum_{m \geq 1} m^{-\frac{32+\ve}{4}} m^{7} (\log m)^{8} < \infty.
\end{eqnarray*}
Thus by the convergence Borel--Cantelli lemma, with full probability only finitely many events $A_{m,v,c}$ occur. Now the argument can be completed in analogy with the proof of Theorem \ref{thm:2}.
%Equation \eqref{eq_v_range} controls the minimal size of $v$, which restricts $S$ to be at most as large as $2^{-32+\ve \log_2 m} = m^{-32+\ve}$.
For $S\le (\log_2 N)^{-(32+\ve)}$, for all $\alpha$ from the generic set, and for all sufficiently large $N$, contained in some range $(2^{m}, 2^{m+1}]$, with
\begin{eqnarray*}
v \ge (32+\ve) \log_2 \log_2 N \ge (32+\ve) \log_2 m,
\end{eqnarray*}
analogously to \eqref{eq_split_v} we obtain
\begin{eqnarray*}
\left| \sum_{m \neq n =1}^N \sum_{j \in \mathbb{Z}} \psi_{S/2} (\alpha x_n - \alpha x_m + j) - (N^2 - N) S^2 \right| & \ll &  \left| \sum_{n=1}^N \sum_{v \geq {(32+\ve)} \log_2 \log_2 N} d_v Y_{n,c_v,v} \right| \\
& \ll & \sum_{v \geq {(32+\ve)} \log_2 \log_2 N} 2^{-v} v^{-2} (c+1) N \\
& \ll & NS \sum_{v \geq {(32+\ve)} \log_2 \log_2 N} v^{-2} \\
& \ll &  (\log \log N)^{-1}  N S  = o(NS).
\end{eqnarray*}
This proves Theorem \ref{thm:Main_thm} in the case of polynomials of degree 2.
\end{proof}

\section{Divisors for polynomial sequences: degree 3 and higher} \label{sec_poly}

As in the previous sections, the key point is to gain control of the variance with respect to $\alpha$. A suitable variance bound will be established in Lemma \ref{prop:VarianceBound_poly_max} below. Generally speaking, the case of polynomials of degree $\geq 3$ is less tedious than the quadratic case, since for polynomials $p$ of degree $\geq 3$ it is extremely unlikely for an integer $u$ to allow more than one representation $u = p(m) - p(n)$ (in contrast to the quadratic case, where most numbers representable in this way have numerous such representations). So the representation numbers can be controlled more efficiently (utilizing deep results from the literature) in comparison with the $d=2$ case, whereas now the divisor structure is a bit more tedious to control (in the $d=2$ case we could switch to a gcd sum over all integers, since most integers have at least one representation $u= p(m) - p(n)$; in contrast, now we have to account for the fact that the set of those $u$ that have a representation $u=p(m)-p(n)$ at all is very sparse within $\mathbb{N}$).\\

We recall a classical result on the number of solutions of polynomial congruences. This can be found for example in Chapter 2.6 of \cite{nzm}. Since in this paper the letter ``$p$'' is reserved for the polynomial, throughout this section we will use the letter ``$q$'' to denote primes.

\begin{lem}[Lagrange's Theorem] \label{lemma_lagrange}
Let $p(x) \in \mathbb{Z}[x]$ be a polynomial of degree $d$. Let $q$ be a prime. Assume that not all coefficients of $p$ are divisible by $q$. Then the congruence
$$
p(x) \equiv 0   \left( \textup{mod} \; q \right)
$$
has at most $d$ solutions.
\end{lem}

With this, we prove:

\begin{lem} \label{lemma_div}
Let $p(x) \in \mathbb{Z}[x]$ a polynomial of degree $d \geq 1$, without constant term. Write $\mathcal{D}_N$ for the difference set
$$
\left\{ p(m) - p(n):\, 1 \leq m \neq n \leq N \right\}.
$$
Assume that there is no prime which divides all coefficients of $p$. Then for any integer $\ell > 1$ we have
$$
\# \left\{n \in \mathcal{D}_N:\, \ell \mid n \right\} \leq \frac{(N + \rad(\ell))^2}{\rad(\ell)}d^{\omega(\ell)}.
$$
Here $\rad(\ell)$ is the radical of $\ell$ (product of distinct prime factors of $\ell$), and $\omega$ is the prime omega function (number of distinct prime factors of $\ell$).
\end{lem}

\begin{proof}
Write $M$ for the smallest integer $\geq N$ which is divisible by $\rad(\ell)$. Then clearly
 \begin{eqnarray*}
 \# \left\{n \in \mathcal{D}_N:\, \ell \mid n \right\} &  \leq & \# \left\{n \in \mathcal{D}_N:\, \rad(\ell) \mid n \right\} \\
 & \leq & \# \left\{n \in \mathcal{D}_M:\, \rad(\ell) \mid n \right\} \\
 & = & \frac{M^2}{\rad(\ell)^2} \# \left\{n \in \mathcal{D}_{\rad(\ell)}:\, \rad(\ell) \mid n \right\} \\
 & \leq & \frac{(N + \rad(\ell))^2}{\rad(\ell)^2} \# \left\{n \in \mathcal{D}_{\rad(\ell)}:\, \rad(\ell) \mid n \right\}.
\end{eqnarray*}
An element $n$ of  $\mathcal{D}_{\rad(\ell)}$ is divisible by $\rad(\ell)$ if and only if there are $1\le n_1,n_2 \le \rad(\ell)$ such that $n = p(n_1) - p(n_2)$ and $p(n_1) - p(n_2) \equiv 0$ (mod $\rad(\ell)$). We thus have to count the number of integers $ 1\le n_1,n_2 \le \rad(\ell)  $ such that $ p(n_1) - p(n_2) \equiv 0$ mod ($\rad(\ell) $). Let $q$ be a prime dividing $\rad(\ell)$.  By Lemma \ref{lemma_lagrange} for any integer $b$ the congruence $p(x) \equiv b$ (mod $q$) has at most $d$ solutions. Applying the lemma with $b=p(n_2)$ for every $n_2$, and noting that there are $q$ choices for $n_2$, we see that the number of solutions of the congruence $p(n_1) - p(n_2) \equiv 0$ (mod $q$) is at most $dq$. By the Chinese Remainder Theorem, this implies that
$$
\# \left\{n \in \mathcal{D}_{\rad(\ell)}:\, \rad(\ell) \mid n \right\} \leq d^{\omega(\rad(\ell))} \rad(\ell),
$$
where $\omega$ is the prime omega function. Clearly, $\omega(\rad(\ell)) = \omega(\ell)$. Thus overall we obtain
$$
\# \left\{n \in \mathcal{D}_N:\, \ell \mid n \right\} \leq \frac{(N + \rad(\ell))^2}{\rad(\ell)}d^{\omega(\ell)}.
$$
\end{proof}

Lemma \ref{lemma_div} allows us to obtain an upper bound for the sums of greatest common divisors, which, as we have seen, play a key role for the variance estimate.

\begin{lem} \label{lem:gcd_lemma}
Let $p(x) \in \mathbb{Z}[x]$ a polynomial of degree $d$. Let $\mathcal{D}_N$ be defined as in the statement of the previous lemma. Then, for any fixed $\varepsilon >0$, and for all $L$ in the range $1 \leq L \leq N$,
\begin{equation} \label{lem_gcd_poly}
\sum_{m,n \in \mathcal{D}_N} \frac{(m,n)}{\max \left\{m,n\right\}} \cdot \mathbf{1} \left(\frac{(m,n)^2}{m n} \geq \frac{1}{L} \right)  \ll_{\ve,d} N^2 L^{\varepsilon}.
\end{equation}
\end{lem}

\begin{proof}
Note that for the statement of this lemma we can assume w.l.o.g.\ that the constant term of $p$ vanishes (since the conclusion of the lemma is only about differences of values of $p$, not about values of $p$ themselves). Furthermore, we can assume w.l.o.g.\ that there is no prime which divides all coefficients of $p$ (since otherwise we could divide $p$ by any such prime without affecting the size of the left-hand side of \eqref{lem_gcd_poly}).\\

Assume that $m \leq n$, so that $\max\{m,n\} = n$. If $\frac{(m,n)}{n} < 1/L$, then $\frac{(m,n)^2}{mn} < 1/L$ and the pair $m,n$ does not contribute anything to \eqref{lem_gcd_poly}. Let $\ell$ be a number in the range $2 \leq \ell \leq L$.  Then
$$
\frac{(m,n)}{n} = \frac{1}{\ell}
$$
is only possible if $\ell$ divides $n$. By Lemma \ref{lemma_div}, there are at most (we use that $\rad(\ell) \leq \ell \leq L \leq N$)
$$
\frac{(N + \rad(\ell))^2}{\rad(\ell)}d^{\omega(\ell)} \leq \frac{4 N^2}{\rad(\ell)}d^{\omega(\ell)}
$$
many elements of $\mathcal{D}_N$ which are divisible by $\ell$. For each of these, there are at most $\ell$ many numbers $m\leq n$ for which $(m,n)/n = 1/\ell$, namely the integer multiples of $n / \ell$. Accordingly,
$$
\sum_{m,n \in \mathcal{D}_N}\frac{(m,n)}{\max \left(m,n\right)} \cdot \mathbf{1} \left(\frac{(m,n)}{\max \left(m,n\right)} \geq \frac{1}{L} \right)  \ll N^2 \sum_{\ell=1}^L \frac{d^{\omega(\ell)}}{\rad(\ell)}.
$$
Thus it remains to prove that
$$
\sum_{\ell=1}^L \frac{d^{\omega(\ell)}}{\rad(\ell)} \ll L^\ve
$$
for any fixed $\ve>0$, which follows from
\begin{eqnarray*}
\sum_{\ell=1}^L \frac{d^{\omega(\ell)}}{L^\ve \rad(\ell)} \leq \sum_{\ell=1}^\infty \frac{d^{\omega(\ell)}}{ \rad(\ell)^{1+\ve}} = \prod_{q \textup{ prime}} \left(1 + \frac{d}{q^{1+\ve}} \right) < \infty.
\end{eqnarray*}
%\begin{eqnarray*}
%\sum_{\ell=1}^L \frac{d^{\omega(\ell)}}{L^\ve \rad(\ell)} &\leq & \sum_{\ell=1}^\infty \frac{d^{\omega(\ell)}}{\ell^\varepsilon \rad(\ell)} \\
%& = & \prod_{q \textup{ prime}} \left(1 + \frac{d}{q^{1+2\ve}} + \frac{d}{q^{1+3\ve}} + \frac{d}{q^{1 + 4 \ve}} \dots \right) \\
%& = & \prod_{q \textup{ prime}} \left(1 + \frac{d}{q^{1+\ve} (q^{\ve}-1)} \right) \\
%& \leq & \prod_{q \textup{ prime}} \left(1 + \frac{c_\ve d}{q^{1+\ve}} \right) < \infty.
%\end{eqnarray*}
\end{proof}

For a given $N$ we define $\mathcal{U}_1$ and $\mathcal{U}_2$ by
$$
 \mathcal{U}_1:= \{ u \geq 1:~\textup{Rep}_{1,N} (u)  = 1\}, \qquad \mathcal{U}_2:= \{ u  \geq 1:~\textup{Rep}_{1,N} (u) \geq 2\}.
$$

\begin{lem} \label{lemma_rep_deg_3}
Let $p(x) \in \mathbb{Z}[x]$ be a polynomial of degree $d \geq 3$.
%Let
%$$
%\textup{Rep}_{1,N} (u) = \# \left\{m, n \leq N:~|p(m) - p(n)| = u \right\}.
%$$
Then
$$
\sum_{u \in \mathcal{U}_2} \left(\textup{Rep}_{1,N} (u) \right)^2 = \mathcal{O} \left( N^{2 - \frac{1}{250}} \right).
$$
\end{lem}

\begin{proof}
The representation function
$$
r(u) := \# \left\{m, n:~p(m) + p(n) = u \right\}
$$
has been studied intensively; note that this function counts representations of integers as \emph{sums} of polynomial values, rather than as differences (which is a less studied problem, but of course closely related). Very  sophisticated upper bounds for $r(u)$ are known, and in particular we have
\begin{equation} \label{eq_lemma_54}
\sum_{1 \leq u \leq x} r (u)^2  = 2 \sum_{1 \leq u \leq x} r(u) + \mathcal{O} \left( x^{\frac{2 - 1/250}{d}} \right) \qquad \text{as $x \to \infty$.}
\end{equation}
There are more precise results in the literature, with error terms of the form $\ll x^{(2-\delta_d)/d}$ for some explicit constants $\delta_d$ depending on the degree $d$ of the polynomial. For our purpose it is sufficient to know that $\delta_d>0$ for all $d \geq 3$. The fact that $\delta_d$ can be chosen e.g.\ as $1/250$ follows from work of Wooley \cite{wooley} ($d=3$), Browning \cite{brown} ($d=4$), and Browning and Heath-Brown \cite{bhb} ($d \geq 5$). Equation \eqref{eq_lemma_54} essentially says that for those $u$ which can be represented in the form $u = p(m) + p(n)$ at all, apart from some very rare exceptions this representation is unique (up to exchanging the roles of $m$ and $n$, which gives the factor $2$ in \eqref{eq_lemma_54}). As a consequence of \eqref{eq_lemma_54} we have, by an application of the inequality $y^2 \leq 3(y^2-2y)$ which is valid for $y \geq 3$,
\begin{equation}
\sum_{u \leq x:\, r(u) \geq 3} r(u)^2 \leq 3 \left( \sum_{1 \leq u \leq x} \left(r (u)^2  - 2 r(u) \right) \right) = \mathcal{O} \left( x^{\frac{2 - 1/250}{d}} \right).
\end{equation}
Now we use the fact that as far as the \emph{second} moment is concerned, the representation functions of sums and of differences can be compared quite efficiently (fortunately second moments of the representation function are sufficient in this section; this is in contrast to the previous section, where it was necessary to work with third moments). Indeed, we have
\begin{eqnarray*}
&  \sum_{u \in \mathcal{U}_2}  \left(\textup{Rep}_{1,N} (u) \right)^2 & \ll \sum_{u \in \mathcal{U}_2}  \textup{Rep}_{1,N} (u) \left( \textup{Rep}_{1,N} (u) - 1 \right)  \\
& & =  \# \Big\{ (n_1,n_2,n_3,n_4) \in \{1,\dots,N\}^4:\, n_1 \neq n_2,\, n_3 \neq n_4 ,\,  \\
 && \qquad \qquad (n_1, n_2)\ne (n_3,n_4) ,\, p(n_1) - p(n_2) = p(n_3) - p(n_4)  \in \mathcal{U}_2  \Big\} \\
&  & \ll   \# \Big\{ (n_1,n_2,n_3,n_4) \in \{1,\dots,N\}^4:\, (n_1, n_3)\ne (n_2,n_4) ,\,  \\  && \qquad \qquad  (n_1, n_2)\ne (n_3,n_4) ,\,  p(n_1) + p(n_4) = p(n_2) + p(n_3) \Big\} \\
& &  \ll \sum_{u \ll N^{d}:\, r(u) \geq 3} r(u)^2 \ll N^{2 - \frac{1}{250}} .
\end{eqnarray*}
%Here the term $N^2$ came from the trivial solutions of $p(n_1) + p(n_4) = p (n_3) + p(n_2)$ of the form $(n_3,n_2) = (n_1,n_4)$. This yields
%$$
%\sum_{u \geq 1} \left(\textup{Rep}_{1,N} (u) \right)^2 = \hlm{\frac{ N^2 }{2} } +  \mathcal{O} \left( N^{2 - \frac{1}{250}} \right),
%$$
as desired.
\end{proof}

\begin{lem} \label{prop:VarianceBound_poly_max}
Let $p(x) \in \mathbb{Z}[x]$ a polynomial of degree $d \geq 3$, and let $x_n = p(n)$ for $n \geq 1$. Let $\varepsilon > 0$. Then for all $v,c$ as specified in \eqref{eq:param_ranges},  we have
$$
\mathbb{E}\left[\left( \max_{1 \leq M \leq N} \left| \sum_{n=1}^N  Y_{n,v,c} \right| \right)^2 \right]  \ll_\varepsilon  2^{-(3-\ve)v} (\log N)^2  (\log \log N)^2 (v+1)(c+1) N^2,
$$
with an implied constant that is independent of $N,v,c$.
\end{lem}

\begin{proof}
By Lemma \ref{lemma_rep_deg_3} we have
\begin{equation} \label{eq_u2_cont}
\sum_{u \in \mathcal{U}_2}  \left( \textup{Rep}_{1,N}(u) \right)^2 \ll N^{2 - \frac{1}{250}}.
\end{equation}
By \eqref{eq:Sum_Yn}, when taking into consideration the decomposition of $\mathbb{N}$ into $\mathcal{U}_1$ and $\mathcal{U}_2$, and using the inequality $(x+y)^2 \leq 2 (x^2 + y^2)$ for $x,y \in \mathbb{R}$, we have
\begin{eqnarray*}
& & \mathbb{E}\left[\left(\sum_{N_1 \leq n \leq N_2} Y_{n,v,c} \right)^2 \right] \ll   \int_0^1 \left( \sum_{u \in \mathcal{U}_1} \textup{Rep}_{N_1,N_2}\left(u\right)  \sum_{j\ne 0} \widehat{f_{v,c}}\left(j\right)e^{2\pi ij\alpha u}\right)^2 d\alpha \\
& +&  \int_0^1 \left( \sum_{u \in \mathcal{U}_2}  \textup{Rep}_{N_1,N_2}\left(u\right)  \sum_{j\ne 0} \widehat{f_{v,c}}\left(j\right)e^{2\pi ij\alpha u} \right)^2 d\alpha.
\end{eqnarray*}
For $u \in \mathcal{U}_1$ we use version of \eqref{eq_variance_upper_b_2} with the condition $ \frac{u_{1} u_{2}}{(u_{1},u_{2})^2} \leq 2^{\frac{9v}{4}}$ replaced by  $\frac{u_{1} u_{2}}{(u_{1},u_{2})^2} \leq 2^{3v}$, and obtain
\begin{eqnarray*}
 & &  \int_0^1 \left( \sum_{u \in \mathcal{U}_1} \textup{Rep}_{N_1,N_2}\left(u\right)  \sum_{j\ne 0} \widehat{f_{v,c}}\left(j\right)e^{2\pi ij\alpha u} \right)^2 d\alpha  \\
& \ll & 2^{-3v}  \sum_{u_{1},u_{2} \in \mathcal{U}_1}\textup{Rep}_{N_1,N_2}\left(u_{1}\right)\textup{Rep}_{N_1,N_2} \left(u_{2}\right)  \times   \nonumber\\
& &\qquad \times \left( (v+1) (c+1) \frac{(u_{1},u_{2})}{\max\{u_{1},u_{2}\}}  \mathbf{1} \left( \frac{u_{1} u_{2}}{(u_{1},u_{2})^2} \leq 2^{3v} \right)  + \frac{(u_1,u_2)^2}{u_1 u_2} \right). \nonumber
\end{eqnarray*}
For $u \in \mathcal{U}_2$ we rather proceed as in the steps leading to \eqref{eq_gcd_arises}, and obtain
\begin{eqnarray*}
& &  \int_0^1 \left( \sum_{u \in \mathcal{U}_2}  \textup{Rep}_{N_1,N_2}\left(u\right)  \sum_{j\ne 0} \widehat{f_{v,c}}\left(j\right)e^{2\pi ij\alpha u} \right)^2 d\alpha \\
& \ll & 2^{-3v}  (v+1) (c+1) \sum_{u_{1},u_{2} \in \mathcal{U}_2}\textup{Rep}_{N_1,N_2}\left(u_{1}\right)\textup{Rep}_{N_1,N_2}\left(u_{2}\right)  \frac{(u_{1} , u_{2})}{\sqrt{u_{1} u_{2}}}.
\end{eqnarray*}
Accordingly, an application of the maximal inequality in Lemma \ref{lemma_mss} yields
$$
\mathbb{E}\left[\left( \max_{1 \leq M \leq N} \left| \sum_{n=1}^N  Y_{n,v,c} \right| \right)^2 \right]   \ll S_1 + S_2,
$$
where
\begin{eqnarray}
S_1 & =& 2^{-3v} (\log N)^2 \sum_{u_{1},u_{2} \in \mathcal{U}_1} \underbrace{\textup{Rep}_{1,N}\left(u_{1}\right)\textup{Rep}_{1,N} \left(u_{2}\right)}_{=1} \times \nonumber\\
& &\qquad  \times \left( (v+1) (c+1) \frac{(u_{1},u_{2})}{\max\{u_{1},u_{2}\}}  \mathbf{1} \left( \frac{u_{1} u_{2}}{(u_{1},u_{2})^2} \leq 2^{3v} \right)  + \frac{(u_1,u_2)^2}{u_1 u_2} \right) \label{eq_s1_terms}
\end{eqnarray}
and
$$
 S_2 = 2^{-3v} (\log N)^2 (v+1)(c+1) \sum_{u_{1},u_{2} \in \mathcal{U}_2}\textup{Rep}_{1,N}\left(u_{1}\right)\textup{Rep}_{1,N}\left(u_{2}\right)  \frac{(u_{1} , u_{2})}{\sqrt{u_{1} u_{2}}}.
$$
We can estimate $S_2$ using the bound for GCD sums with parameter 1/2, similar to the calculations leading to Lemma \ref{prop:VarianceBound_coarse_2}, where instead of $E(1,N)$ we have to put $\sum_{u \in \mathcal{U}_2} \left( \textup{Rep}_{1,N} (u) \right)^2$, and obtain
$$
S_2 \ll 2^{-3v} (v+1)(c+1) N^\varepsilon N^{2 - \frac{1}{250}} \ll 2^{-3v} (v+1)(c+1) N^{2 - \frac{1}{500}}
$$
(when choosing $\varepsilon = 1/500$) as a consequence of \eqref{eq_u2_cont}. Essentially, what we used here is that the estimate for the GCD sum loses a factor $N^\varepsilon$, but we won a factor $N^{\frac{1}{250}}$ from the fact that only very few $u$ are in $\mathcal{U}_2$, so that overall the contribution of $u \in \mathcal{U}_2$ is negligible. \\

The relevant contribution comes from $S_1$. The final term in line \eqref{eq_s1_terms}, i.e.\ the term $(u_1,u_2)^2 /(u_1 u_2)$, leads to a GCD sum with parameter 1, which can be efficiently estimated and is of size
$$
2^{-3v} (\log N)^2 (\log \log N)^2 \sum_{u \in \mathcal{U}_1}  \left( \textup{Rep}_{1,N}(u) \right)^2 \ll 2^{-3v} (\log N)^2 (\log \log N)^2 N^2.
$$
The rest of the contribution comes from
\begin{eqnarray*}
& &  2^{-3v} (\log N)^2 \sum_{u_{1},u_{2} \in \mathcal{U}_1} ( v+1) (c+1) \frac{(u_{1},u_{2})}{\max\{u_{1},u_{2}\}}  \mathbf{1} \left( \frac{u_{1} u_{2}}{(u_{1},u_{2})^2} \leq 2^{3v} \right).
\end{eqnarray*}
Clearly $\mathcal{U}_1$ is a subset of the difference set $\mathcal{D}_N$ which was defined in the statement of Lemma \ref{lemma_div}. Thus, an application of  Lemma \ref{lem:gcd_lemma} yields
\begin{eqnarray*}
& & 2^{-3v}(\log N)^2 \sum_{u_{1},u_{2} \in \mathcal{U}_1} ( v+1)(c+1) \frac{(u_{1},u_{2})}{\max\{u_{1},u_{2}\}}  \mathbf{1} \left( \frac{u_{1} u_{2}}{(u_{1},u_{2})^2} \leq 2^{3v} \right) \\
& \ll &  2^{-(3-\varepsilon)v} (\log N)^2 (v+1) (c+1) N^2,
\end{eqnarray*}
where $\varepsilon>0$ is arbitrary but fixed. Overall this yields
$$
S_1 \ll 2^{-(3-\varepsilon)v} (\log N)^2 (\log \log N)^2 (v+1) (c+1) N^2 \qquad \text{and} \qquad S_2 \ll 2^{-3v} (v+1) (c+1) N^{2 - \frac{1}{500}}.
$$
Thus,
%We have $S_2 \ll S_1$ if $\varepsilon$ is sufficiently small, and so
$$
\mathbb{E}\left[\left( \max_{1 \leq M \leq N} \left| \sum_{n=1}^N  Y_{n,v,c} \right| \right)^2 \right] \ll 2^{-(3-\ve)v}  (\log N)^2 (\log \log N)^2 (v+1) (c+1) N^2,
$$
as claimed.
\end{proof}

\begin{proof}[Proof of Theorem \ref{thm:Main_thm} in the case of degree $3$ or higher]
 The proof in the case of polynomials of degree $d \geq 3$ is very similar to the proof for the case $d=2$, which we gave in Section \ref{sec_poly_deg2}. The main difference is that now we have the stronger variance estimate in Lemma \ref{prop:VarianceBound_poly_max} rather than Lemma \ref{prop:VarianceBound_deg2}, which allows us to obtain a quantitatively stronger result.\\

 We define
 $$
\mathcal{A}_{m,v,c}=\left\{ \alpha\in\left[0,1\right]:\,\max_{2^{m} < M \leq 2^{m+1}} \left| \sum_{n=1}^M Y_{n,v,c}\right| > 2^{-v} v^{-2} (c+1) 2^{m} \right\},
$$
for the range
\begin{equation*}
v \ge (3 + 4 \ve) \log_2 m, \qquad 0 \leq c < 2^v.
\end{equation*}
By Chebyshev's inequality and Lemma \ref{prop:VarianceBound_poly_max} we have
$$
\textup{meas}(\mathcal{A}_{m,v,c}) \ll 2^{-v + \ve v} m^2 (\log m)^2 (c+1) v^{5} (c+1)^{-2} \ll 2^{-v + \ve v} v^{5} m^2 (\log m)^2 (c+1)^{-1}.
$$Accordingly,
\begin{eqnarray*}
& &  \sum_{m \geq 1} \, \sum_{v \ge (3 + 4\ve) \log_2 m} \, \sum_{0 \leq c < 2^v} \textup{meas} \left( \mathcal{A}_{m,v,c} \right)  \\
& \ll &  \sum_{m \geq 1} \, \sum_{ v \geq  (3 + 4 \ve) \log_2 m} \, \sum_{0 \leq c < 2^v} 2^{- v(1-\ve)} v^{5} m^2 (\log m)^2 (c+1)^{-1}\\
& \ll &  \sum_{m \geq 1}  \, \sum_{ v \geq  (3 + 4 \ve) \log_2 m} 2^{- v(1-\ve)} v^6 m^2 (\log m)^2 \\
& \ll &  \sum_{m \geq 1}  m^2 m^{-(3 + 4\ve)(1-\ve)}  (\log m)^8 \\
& \ll &  \sum_{m \geq 1}  m^{-1 - \ve + 4 \ve^2}  (\log m)^8  < \infty,
\end{eqnarray*}
assuming that $\ve$ was chosen sufficiently small. Thus by the convergence Borel--Cantelli lemma, with full probability only finitely many events $A_{m,v,c}$ occur. Now the argument can be completed as in the previous proofs.
%The minimal admissible size of $v$ now means that we can handle $S$ up to at most  $2^{-(3+4\ve) \log_2 m} = m^{-3 + 4\ve}$.
For $S\le (\log_2 N)^{-(3+4\ve)}$, for all $\alpha$ from the generic set, and for all sufficiently large $N$ contained in some range $(2^{m}, 2^{m+1}]$, with
\begin{eqnarray*}
v \ge (3+4\ve) \log_2 \log_2 N \ge (3+4\ve) \log_2 m,
\end{eqnarray*}
analogously to \eqref{eq_split_v} we obtain
\begin{eqnarray*}
\left| \sum_{m \neq n =1}^N \sum_{j \in \mathbb{Z}} \psi_{S/2} (\alpha x_n - \alpha x_m + j) - N^2 S^2 \right| & \ll &   \left| \sum_{n=1}^N \sum_{ v \geq  (3 + 4 \ve) \log_2 \log_2 N} d_v Y_{n,c_v,v} \right| \\
& \ll & \sum_{v \geq (3 + 4\ve) \log_2 \log_2 N} 2^{-v} v^{-2} (c+1) N  \\
& \ll & NS \sum_{ v \geq  (3 + 4 \ve) \log_2 \log_2 N} v^{-2} \\
& \ll &  (\log \log N)^{-1}  N S = o(NS),
\end{eqnarray*}
which proves Theorem \ref{thm:Main_thm} in the case of polynomials of degree 3 and higher.
\end{proof}

\section{Kronecker sequences and random sequences}    \label{sec_final}

Proposition \ref{prop1} concerns the number variance of the sequence $(\alpha n)_{n=1}^\infty$. This sequence is a very classical object in the theory of uniform distribution modulo 1, and is called ``Kronecker sequence'' in this context. It is well-known that the distribution of $(\alpha n)_{n=1}^\infty$ mod 1 is closely connected with the Diophantine approximation properties of $\alpha$, and in particular with the size of the partial quotients in the continued fraction expansion of $\alpha$. Since continued fractions only pertain to a side aspect of this paper, we do not give a more detailed introduction into the subject here, and instead refer to the literature: see \cite{karp,khin,rock} for basic information on continued fractions and Diophantine approximation, \cite{dts,kn} for the connection with the distribution of $(\alpha n)_{n=1}^\infty$ mod 1, and \cite{allm,goda,ls_neg,wskill} for recent work exploring the relation between the gap structure of sequences mod 1 and their pseudorandom behavior with respect to ``local'' test statistics.\\

Before proving Proposition \ref{prop1}, we note that for rational $\alpha$ the number variance of $(\alpha n)_{n=1}^\infty$ clearly also fails to be Poissonian, since in this case there are only finitely many possible values of $\alpha n$ mod 1, and one can easily calculate that this implies $V(N,S,\alpha) \gg_\alpha N^2 S$ (for sufficiently large $N$, and as long as $S \leq 1/2$, say). Accordingly,  in the case of rational $\alpha$ the number variance blows up, and is asymptotically by far too large. (As a side remark, the same happens for rational $\alpha$ and any other integer-valued sequence $(x_n)_{n=1}^\infty$, which implies that rational values of $\alpha$ can never belong to the generic sets in the conclusion of Theorems \ref{thm:Main_thm} resp.\ \ref{thm:2}). Thus for rational $\alpha$ the number variance of $(\alpha n)_{n = 1}^\infty$ is too large; next we will show that for all irrational $\alpha$ the number variance of $(\alpha n)_{n=1}^\infty$ is too small for infinitely many $N$, which implies that actually there is no $\alpha$ at all for which the number variance of $(\alpha n)_{n=1}^\infty$ is Poissonian.

\begin{proof}[Proof of Proposition \ref{prop1}] Let some irrational $\alpha$ be given. It is convenient to write $V(N,S,\alpha)$ in the form
$$
V(N,S,\alpha)= \int_0^1 \left( S_N(S,y) - NS \right)^2 dy.
$$
By Dirichlet's approximation theorem there exist infinitely many $N$ for which there is a (coprime)  integer $p$ such that
$$
\left| \alpha - \frac{p}{N} \right| < \frac{1}{N^2}.
$$
Thus $n \alpha$ is well approximated by $n p/N$, where it is crucial that by co-primality the sequence $(n p)_{n=1}^N$ runs through a complete residue system mod $N$. Since
$$
\left|\alpha n - \frac{n p}{N} \right| < \frac{1}{N}, \qquad 1 \leq n \leq N,
$$
it is easily seen that for any interval $S$ we have $\left|S_N(S,y) - NS\right| \leq 3$, uniformly in $S$ and in $y$, and accordingly
$$
V(N,S,\alpha)= \int_0^1 \left( S_N(S,y) - NS \right)^2 dy \leq 9,
$$
as claimed. Thus $V(N,S,\alpha) = NS + o(NS)$ is not possible when $NS \to \infty$.
\end{proof}

\begin{proof}[Proof of Proposition \ref{prop2}]
Before we prove Proposition \ref{prop2}, we briefly explain the heuristics why $S \approx (\log \log N)^{-1}$ should be the critical threshold where the asymptotics $V(N,S) \sim NS$ starts to break down (almost surely). Consider an interval $I$ of length $S$, where $S = a L$ for $L = 1 / \log \log N$ and for some suitable constant $a$. The indicator $\mathbf{1}_I(X_n)$ is (for every $n$) a random variable with variance $S(1-S) \approx S$. Thus by the law of the iterated logarithm, the number of points $X_1, \dots, X_N$ contained in $I$ should deviate from the expected number by $\pm \sqrt{ 2  S N \log \log N} = \sqrt{2aN}$ for infinitely many $N$, almost surely. In other words, for infinitely many $N$ the ``local'' contribution to the number variance coming from this particular interval is $2aN$, which is very large. If we shift the location of the interval $I$ by $y$, where $y$ is small in comparison to the length of $I$ (say, $y = b S$ for a small constant $b$), then the deviation between the expected number of points and the actual number of points in the shifted interval should still remain large (only a bit less than $\sqrt{2aN}$, if $b$ is small in comparison with $a$). Thus we can integrate over a range of length $bS$ for this shift parameter $y$, so that the total number variance is essentially as large as $V(N,S) \geq 2aN bS$. This exceeds $NS$ if $a$ and $b$ are chosen such that $ab > 1$, and thus the number variance is too large to match with Poissonian behavior. This is the general strategy, but to actually carry out the proof are there several details that need to be taken care of.\\

The proof of Proposition \ref{prop2} could be given using only basic estimates for the tail probabilities of Bernoulli random variables (such as those in Khintchine's classical paper on the law of the iterated logarithm \cite{khin_LIL}), together with a dyadic decomposition of the location parameter which allows to handle shifted intervals, but this would result in a long and very technical proof. At the other extreme, one could aim to apply Talagrand's sophisticated ``generic chaining'' machinery \cite{tala_book} and reduce it to the geometrically very simple situation that we are dealing with in the present setup. A third possibility, which we will follow here, is to use the Koml\'{o}s--Major--Tusn\'{a}dy theorem \cite{kom}, which states the following (in the following statement, $a,b,c$ denote suitable positive constants): Let $X_1, X_2, \dots$ be i.i.d.\ random variables having uniform distribution on $[0,1]$. Then (under the technical assumption that the underlying probability space is large enough, such that it allows the construction of a Brownian motion) there exists a Brownian bridge $B(t),~0 \leq t \leq 1$, such that for all $x > 0$
$$
\mathbb{P} \left( \sup_{0 \leq t \leq 1} \left| \sum_{n=1}^N \mathbf{1}_{[0,t)} (X_n) - N t  - \sqrt{N} B(t) \right|  > a \log N + x  \right) \leq b e^{-c x}.
$$
We will use this in the form
\begin{equation} \label{equ_with_d}
\mathbb{P} \left( \sup_{0 \leq t \leq 1} \left|\sum_{n=1}^N \mathbf{1}_{[0,t)} (X_n) - N t  - \sqrt{N} B(t) \right|  > d \log N  \right) \leq N^{-2}
\end{equation}
for some suitable constant $d$, and all sufficiently large $N$. We write $\mathbf{I}_A$ for the indicator function of an interval $A$ that has been centered to have average zero, i.e.\ $\mathbf{I}_A (x) = \mathbf{1}_A (x) - \textup{length} (A)$. For $m \geq 1$, we set $N_m
= 10000^m$, and write $L = L_m = (\log \log N_m)^{-1}$. We define the events
\begin{eqnarray*}
A_m & = &\left\{\sum_{n=N_{m-1}+1}^{N_m} \mathbf{I}_{[0, 35 L)} (X_n) \geq \sqrt{\frac{699}{10}} \sqrt{N_m - N_{m-1}} -  d \log N_m \right\},\\
B_m & = & \left\{\sup_{0 \leq t \leq 2L} \left| \sum_{n=1}^{N_m} \mathbf{I}_{[0, t)} (X_n)  \right| \geq \sqrt{\frac{41}{10}} \sqrt{N_m} + d \log N_m \right\},\\
C_m & = & \left\{\sup_{0 \leq t \leq 2L} \left| \sum_{n=1}^{N_m} \mathbf{I}_{[35L , 35L + t)} (X_n)  \right| \geq \sqrt{\frac{41}{10}} \sqrt{N_m} + d \log N_m \right\}, \\
D_m & = &\left\{ \left| \sum_{n=1}^{N_{m-1}} \mathbf{I}_{[0, 35 L)} (X_n) \right| \geq \sqrt{\frac{701}{10}}\sqrt{N_{m-1}} +  d \log N_m \right\},\\
\end{eqnarray*}
where $d$ is the constant from \eqref{equ_with_d}. Heuristically, $A_m$ states that there is an interval whose contribution to the number variance is large, and $B_m$ and $C_m$ allow us to shift this interval by a (small) shift $t$, while keeping most of this large contribution to the number variance. The sum in the definition of $A_m$ starts at $N_{m-1}+1$ rather than 1 since we need the sets $A_m$ to be stochastically independent, in order to apply the divergence Borel--Cantelli lemma (in contrast, to $B_m,C_m,D_m$, where we apply the convergence Borel--Cantelli lemma, which does not require independence); since the summation in $A_m$ only starts at $N_{m-1}+1$, the set $D_m$ controls the sum over the remaining part of the index set.\\

By \eqref{equ_with_d}, for sufficiently large $m$ we have
\begin{equation} \label{equ_with_d_2}
\mathbb{P} (A_m) \geq \mathbb{P} \left( B (35 L) \geq \sqrt{\frac{699}{10}}\right) - (N_m - N_{m-1})^{-2}.
\end{equation}
The last term in this equation is of negligible size. The distribution of the Brownian bridge at ``time'' $35L$  is that of a normal random variable with mean zero and variance $35L(1-35L)$. For large $m$ the factor $(1-35 L)$ is negligible (note that $L \to 0$ as $m \to \infty$), and since the tail probabilities of a normal random variable decay very roughly as $e^{-x^2/(2 \sigma^2)}$, we have
\begin{equation} \label{equ_with_d_3}
\mathbb{P} \left( B (35 L) \geq \sqrt{\frac{699}{10}} \right) \approx e^{\frac{-699}{700 L}} \approx  m^{-699/700}.
\end{equation}
Since Proposition \ref{prop2} pertains only to a side aspect of this paper, we allow ourselves to be a bit sketchy here; we write the symbol ``$\approx$'' in a vague sense, meaning ``essentially of a certain order, up to insignificant multiplicative terms''.  The key point is that, as a consequence of \eqref{equ_with_d_2} and  \eqref{equ_with_d_3}, we can establish that
$$
\sum_{m=1}^\infty \mathbb{P} (A_m) = + \infty,
$$
which will allow us an application of the divergence Borel-Cantelli lemma. For the sets $B_m$ we have
$$
\mathbb{P} (B_m) \leq \mathbb{P} \left( \sup_{0 \leq t \leq 2L} \left| B(t) \right| \geq \sqrt{\frac{41}{10}} \right) + N_m^{-2}
$$
for sufficiently large $m$. The tail probabilities of the supremum of a Brownian bridge are very similar to the tail probabilities at the endpoint of the time window under consideration; that is, in our setting,
\begin{equation} \label{equ_b_l}
\mathbb{P} \left( \sup_{0 \leq t \leq 2L} \left| B(t) \right| \geq \sqrt{\frac{41}{10}} \right) \approx \mathbb{P} \left(\left| B(2L) \right| \geq \sqrt{\frac{41}{10}} \right),
\end{equation}
where again we omit multiplicative factors of insignificant size; relation \eqref{equ_b_l} can be argued by using the fact that a Brownian bridge $B(t)$ can be written as $B(t) = W(t) - t W(1)$, where $W$ is a standard Brownian motion (Wiener process), and employing the reflection principle for the Brownian motion which asserts that the distribution of $\sup_{0 \leq t \leq s} W(t)$ is essentially the same as the distribution of $|W(s)|$ (up to a factor 2). Accordingly, we have
$$
\mathbb{P} (B_m) \approx e^{\frac{-41}{40 L}} \approx m^{-41/40},
$$
and thus can establish
$$
\sum_{m=1}^\infty \mathbb{P} (B_m) < + \infty.
$$
By rotational invariance of $X_1, X_2, \dots$, we have $\mathbb{P} (C_m) = \mathbb{P} (B_m)$, and thus
$$
\sum_{m=1}^\infty \mathbb{P} (C_m) < + \infty
$$
as well. We can estimate the probabilities for $D_m$ similar to those of $A_m$; the fact that the number 699/10 has been replaced by 701/10 makes the difference between convergence and divergence of the series of tail probabilities, so that one obtains
$$
\sum_{m=1}^\infty \mathbb{P} (D_m) < + \infty.
$$
Accordingly, from the (divergence resp.\ convergence) Borel--Cantelli lemma we can conclude that with probability one infinitely many events $A_m$ occur, but only finitely many events $B_m, C_m, D_m$ occur. Note that the application of the divergence Borel--Cantelli lemma is indeed admissible, since the sets $A_m, ~m \geq 1,$ were constructed in such a way that they are stochastically independent.\\

Now assume that for some $m$ the event $A_m$ has occurred, but none of the events $B_m,C_m,D_m$ has occurred. Then for any $t \in [0,2L)$ we have, using that $N_m - N_{m-1} = \frac{9999}{10000} N_m$ and $N_{m-1} = \frac{N_m}{10000}$,
\begin{eqnarray*}
\sum_{n=1}^{N_m} \mathbf{I}_{[t,t+35 L)} (X_n) & = & \sum_{n=N_{m-1}+ 1}^{N_m} \mathbf{I}_{[0,35 L)} (X_n)  +  \sum_{n=1}^{N_{m-1}} \mathbf{I}_{[0,35 L)} (X_n) \\
 & & \qquad - \sum_{n=1}^{N_m} \mathbf{I}_{[0, t)} (X_n) + \sum_{n=1}^{N_m} \mathbf{I}_{[35 L,35 L + t)} (X_n)\\
& \geq &\underbrace{\left(\sqrt{\frac{699}{10}} \sqrt{\frac{9999}{10000}} - \sqrt{\frac{701}{100000}} -2 \sqrt{\frac{41}{10}} \right)}_{> \frac{42}{10}} \sqrt{N_m} - 4 d \log N_m.
\end{eqnarray*}
The term $4d \log N_m$ is  asymptotically negligible, and thus we have
$$
\sum_{n=1}^{N_m} \mathbf{I}_{[t,t+35 L)} (X_n) \geq \frac{42N_m}{10}
$$
uniformly for all $t$ in $0 \leq t \leq 2L$ (provided that $m$ is sufficiently large). Now write $S = 35L$, so that $2L = 2S/35$. Then for such $N = N_m$ as above we have
\begin{eqnarray*}
V(N,S) & = & \int_0^1 \left( \sum_{n=1}^N \mathbf{I}_{[t,t+35 L)} (X_n) \right)^2 dt \\
& \geq & \int_0^{2L}  \left(\sum_{n=1}^N \mathbf{I}_{[t,t+35 L)} (X_n) \right)^2 dt \\
& \geq &  \int_0^{2S/35}  \left( \frac{42N}{10} \right)^2 dt \\
& = & \frac{126}{125} N S.
\end{eqnarray*}
Thus, almost surely, there exist infinitely many $N$ for which $V(N,S) \geq \frac{126}{125} NS$, which means that $V(N,S) = NS + o(NS)$ cannot be true (almost surely) as $N \to \infty$. This proves the proposition.
\end{proof}

\section{Acknowledgements}

The first author was supported by the Austrian Science Fund (FWF), projects 10.55776/I4945, 10.55776/I5554, 10.55776/P34763 and 10.55776/P35322. The second author was supported by the ISRAEL SCIENCE FOUNDATION (Grant No. 1881/20). The authors want to thank Ze\'ev Rudnick for introducing them to the problem and for many stimulating discussions.

\bibliography{D}
\bibliographystyle{abbrv}

\end{document}